\documentclass[reqno]{amsart}
\usepackage{amsthm, amscd, amsfonts, amssymb, graphicx,tikz, color, environ}

\usepackage{paralist}
\usepackage{cite}
\usepackage{bigints}
\usepackage{dsfont}
\usepackage{empheq}
\usepackage{amssymb}
\usepackage{cases}
\usepackage{enumitem}
\allowdisplaybreaks
\usepackage{algpseudocode}
\usepackage{caption}
\usepackage{booktabs}
\usepackage{bbold}
\usepackage{float}
\usepackage{subcaption}
\usepackage[ruled,vlined,linesnumbered,noresetcount]{algorithm2e}

\usepackage{hyperref}
\usepackage[top=2.7cm, bottom=2.7cm, left=3cm, right=3cm]{geometry}

\theoremstyle{plain}
\newtheorem{theorem}{Theorem}[section]
\newtheorem{example}{Example}

\newtheorem{lemma}[theorem]{Lemma}

\newtheorem{definition}[theorem]{Definition}
\theoremstyle{remark}
\newtheorem{remark}[theorem]{Remark}
\numberwithin{equation}{section}

\title[Backward problem for a degenerate Hamilton-Jacobi equation] 
{Backward problem for a degenerate viscous Hamilton-Jacobi equation: stability and numerical identification}

\author{S. E. Chorfi}
\address{S. E. Chorfi, L. Maniar, Cadi Ayyad University, UCA, Faculty of Sciences Semlalia, Laboratory of Mathematics, Modeling and Automatic Systems, B.P. 2390, Marrakesh, Morocco}
\email{s.chorfi@uca.ac.ma, maniar@uca.ac.ma}

\author{A. Habbal}
\address{A. Habbal, M. Jahid, L. Maniar, A. Ratnani\\ The UM6P-Vanguard Center, University Mohammed VI Polytechnic, Rabat campus, Morocco}
\email{Abderrahmane.Habbal@um6p.ma, Meryeme.JAHID@um6p.ma, Lahcen.Maniar@um6p.ma, Ahmed.Ratnani@um6p.ma}

\address{A. Habbal, University of Côte d’Azur, Inria, LJAD, Parc Valrose 06108 Nice, France}
\email{abderrahmane.habbal@univ-cotedazur.fr}

\author{M. Jahid}

\author{L. Maniar}

\author{A. Ratnani} 

\keywords{Degenerate Hamilton-Jacobi equation, Backward problem, Carleman estimate, Conjugate Gradient, Van Cittert iteration}
 
\makeatletter 
\@namedef{subjclassname@2020}{%
  \textup{2020} Mathematics Subject Classification}
\makeatother

\subjclass[2020]{Primary: 35R30, 35K65, 35F21; Secondary: 35R25, 65M30}

\begin{document}
\begin{abstract}
This work is devoted to the analysis of the backward problem for a viscous Hamilton-Jacobi equation with degenerate diffusion and a general Hamiltonian that is not necessarily quadratic. First, we focus on linear degenerate parabolic equations in the nondivergence setting. We prove the conditional stability of the backward problem using Carleman estimates. Then, by a linearization technique, we prove similar results for the nonlinear viscous Hamilton-Jacobi equation. Regarding numerical identification, we first investigate the linear degenerate equation with noisy data using the adjoint state method, combined with a Conjugate Gradient algorithm, to solve the associated minimization problem. Finally, the numerical identification for the nonlinear viscous Hamilton-Jacobi equation is investigated by the Van Cittert iteration. Numerical tests are presented to show the performance of the proposed algorithms.
\end{abstract}

\maketitle

\section{Introduction}
Degenerate viscous Hamilton-Jacobi equations naturally arise in a range of applications, including optimal control, front propagation, and mean-field game theory, where the diffusion coefficient may vanish or become singular in certain parts of the domain. Backward problems, which involve recovering past states of a system from final-time data, are particularly challenging in this degenerate setting due to the loss of regularity in addition to ill-posedness. Such problems are relevant in many areas of applied science and engineering, including quantum mechanics, stochastic control, game theory, and finance; see, e.g., \cite{BI99} and the references therein. Despite their importance, backward problems for degenerate viscous Hamilton-Jacobi equations have received limited attention in the literature, motivating the present study.

This work investigates the stability of the backward problem and numerical identification for a degenerate Viscous Hamilton-Jacobi (VHJ for short) equation in the one-dimensional case. Throughout the paper, we consider the following notations
$$\Omega :=(0,1),\qquad Q :=\Omega\times (0,T), \qquad \Sigma:=\left\lbrace 0,1\right\rbrace \times(0, T),$$
where $T>0$ is the terminal time.

To motivate our study, let us start with the simple viscous Hamilton-Jacobi equation 
\begin{equation}
\left\{
\begin{aligned}
& u_{t}(x,t) - a(x) u_{xx}(x,t)+ \frac{1}{q}\vert u_{x}(x,t)\vert^q= 0, &&(x,t)\in Q,\\
&u(x,t)=0,\hspace{6.5cm} &&(x,t)\in \Sigma,\\
&u(x,0)=f(x), && x\in \Omega,
\end{aligned}\label{HJBEQ}
\right.
\end{equation}
where we assume hereafter that $q \geqslant 1$ and the diffusion coefficient degenerates (i.e., vanishes) at the endpoints $x=0$ and $x=1$. More precisely, we introduce the following assumption:\\
\noindent\textbf{Assumption I.}
\begin{itemize}
    \item[(i)] $a\in C([0,1])$ such that $a(x)>0$ for all $x\in \Omega$;
    \item[(ii)] $a(0)=a(1)=0$.
\end{itemize}
\bigskip

\noindent{\bf Backward problem.}\\
We seek to determine unknown initial or intermediate states from the final measurement in equations such as \eqref{HJBEQ}, that is, for $t_0\in [0,T)$, recover $u(\cdot,t_0)$ from given measured data of $u(\cdot,T)$. The aim of this paper is twofold:
\begin{itemize}
    \item[(i)] Theoretical: prove conditional stability by Carleman estimates.
    \item[(ii)] Numerical: design efficient algorithms for numerical identification.
\end{itemize}

The solution to the VHJ equation \eqref{HJBEQ} can be interpreted, by the Dynamic Programming Principle, as the value function associated with a stochastic optimal control problem (see e.g. \cite{Barles04, zua11}), i.e.,
$$u(x,t)=\sup_{\alpha_t \in \mathcal{A}}\mathbb{E}_x\left\lbrace f(X_t)\mathds{1}_{t<\tau_x}+\int_0^{t \wedge\tau_x}c_q\vert \alpha_s\vert^{\frac{q}{q-1}}ds\right\rbrace,$$
assuming here $q>1$ for simplicity, where $\mathbb{E}_x$ denotes the conditional expectation by the random event $\{X_0=x\},$ $c_q$ serves as a normalization factor and $X_t$ is the controlled process governed by the stochastic differential equation
\begin{equation}
\left\{
\begin{aligned}
& dX_{t}=\alpha_tdt+  \sigma(X_t)dB_t,\\
&X_0=x\in \Omega,
\end{aligned}\label{sde}
\right.
\end{equation}
with $B_t$ is a one-dimensional standard Brownian motion (on a complete filtered probability space), $(\alpha_s)_{s\geqslant 0}$ is a control belonging to a set $\mathcal{A}$ of admissible processes, and $\tau_x$ denotes the exit time from $\Omega$:
$$\tau_x=\inf\left\lbrace t>0\colon \; X_t\notin \Omega\right\rbrace.$$ The diffusion coefficient $\sigma$ in \eqref{sde} can vanish at some points, and so can $ a(x)= \frac{1}{2}\sigma^2(x)$. An interesting model is given by $a(x)=x^\mu (1-x)^\nu, \; \mu,\nu >0$, which is, for instance, related to the Wright-Fisher diffusion process used in population genetics. See \cite[Examples, p.~48]{Ma68} and also \cite[Ch.~VIII, \S 8]{Sh92} for multidimensional models. The VHJ equation also appears in some physical models of surface growth where it is linked to the Kardar-Parisi-Zhang (KPZ) equation for $q=2$, see, e.g., \cite{Kardar1986}.

Regarding the existing literature on the VHJ equation in the nondegenerate case, several works have investigated the well-posedness in the whole space $\mathbb{R}^N,$ either in the classical sense; see, e.g., \cite{AmourBenArtzi1998,BenArtzi1992} and \cite{BenArtzi2002}, or in the sense of distributions (see, e.g., \cite{BenArtziSoupletWeissler2002}). In the case of a bounded domain $\Omega\subset \mathbb{R}^N,$ under appropriate assumptions on the initial and boundary data, the solution exists on a time interval $[0,T^*)$. Moreover, as $t\to T^*,$ the gradient blows up at the boundary while the solution is bounded (see \cite{FilaLieberman1994, Souplet2002}).

Degenerate parabolic operators, whether in divergence or nondivergence form, have been the subject of numerous studies in the literature. Many problems in physics, biology, and economics naturally involve them. For controllability problems, we refer to \cite{CFR08} where the authors proved the null controllability for degenerate parabolic equations in nondivergence form in one dimension. Then they extended this result to the semilinear problem. In \cite{Alab2006}, the null controllability of degenerate linear and semilinear parabolic equations in divergence form has been proven. The authors in \cite{CR06} have studied the regional null controllability of semilinear degenerate parabolic equations with nonlinearities involving the first derivative. In this case, the classical null controllability results are invalid due to degeneracy. In \cite{Ait11}, the degenerate semilinear system was investigated with two distinct degeneracies, $a_1$ and $a_2$. They have proven the null controllability results using linearization. Moreover, in \cite{Atifi17}, the authors study the numerical null controllability of a degenerate and singular parabolic problem in divergence form. We also mention some results on stability estimates for inverse problems. In \cite{CR10}, the Lipschitz stability for an inverse source problem using Carleman estimates has been established. Furthermore, in \cite{Tort10}, the authors studied the determination of source terms in a degenerate parabolic equation from the knowledge of a locally distributed observation. The authors of \cite{BOUTAA14} have proven the Lipschitz stability for an inverse problem of linear degenerate parabolic systems with one force in the interior degeneracy case.

Moreover, backward problems for nondegenerate (uniformly elliptic) parabolic equations have been widely investigated in numerous papers; we refer to \cite{Kl06} and the references therein for a method based on Carleman estimates, and to \cite{Ya} for a survey, particularly for nondegenerate equations. For degenerate equations, we refer to \cite{Ka15} for Carleman estimates. Recently, in \cite{CY23}, a degenerate parabolic equation in divergence form has been studied for higher dimensions by the Carleman estimates method. We emphasize that the backward problems we consider for degenerate viscous Hamilton-Jacobi equations are relevant to the study of degenerate mean-field game systems; see, e.g., the recent work \cite{CHJMR25}. As for degenerate problems from time-integral measurements, we refer to \cite{Kam20} and the cited bibliography. Recently, special attention has been devoted to backward problems (inverse design) for VHJ equations; we refer to \cite{EZ20, CP20} for first-order HJ equations. In \cite{CP24}, a related problem has been studied, taking into account localizations. The recent paper \cite{EZ23} investigates the reachable set for a multi-dimensional case. However, the literature on backward degenerate problems in nondivergence settings and second-order VHJ equations is scarce, and most of the works impose strong assumptions on the degenerate coefficient, which entails limited applicability.

In this paper, we first prove some conditional stability estimates for corresponding backward degenerate problems using Carleman estimates. Furthermore, we investigate the numerical identification of initial data from noisy final-time data. To this purpose, we start with the linear part of the VHJ equation, extending the methodology used in \cite{Has17}. We utilize the adjoint-state methodology with a gradient-type iteration to approximate the solution. Among recent studies, the paper \cite{Chorfi23} studies the identification of initial temperatures in the heat equation with dynamic boundary conditions. In the same spirit, we can refer to \cite{Atifi20} for theoretical and numerical analysis of a backward problem involving the identification of the initial data in a 2D degenerate parabolic equation. We also refer to \cite{Hassan2007} and \cite{AIT22}, where a similar approach was applied to inverse source problems. For more information, we recommend the excellent book \cite{Has17} that contains more details about this approach. Moreover, we consider the numerical identification of initial data for the nonlinear VHJ equation using Van Cittert's iteration, known in image reconstruction. The latter was employed for the first time in \cite{Carasso} for backward problems for linear and nonlinear parabolic equations with nondegenerate diffusion. Here, we extend the existing analysis to a viscous Hamilton-Jacobi equation with degenerate diffusion.

The following sections of this paper are organized as follows. In Section \ref{sec2}, we establish conditional stability for a linear degenerate equation in a nondivergence setting via Carleman estimates. In Section \ref{sec3}, we prove the conditional stability for a nonlinear VHJ equation following a similar strategy. Then Section \ref{sec4} is devoted to identifying the initial data for a linearized equation using the adjoint methodology combined with a Conjugate Gradient algorithm. Section \ref{sec5} presents the algorithmic implementation of the gradient formula used to numerically identify the initial data in degenerate parabolic equations. For the nonlinear viscous Hamilton-Jacobi equation, the initial datum is recovered through Van Cittert's iterative method. Finally, the conclusion and final comments will be mentioned in Section \ref{sec6}.
\section{Linear degenerate parabolic equation}\label{sec2}
In this section, we will prove a Carleman estimate for the linear part of the VHJ equation \eqref{HJBEQ}. This will be the key ingredient to tackle the nonlinear VHJ equation.

First, we succinctly review key results on well-posedness and regularity for degenerate parabolic equations in the nondivergence setting. We refer to \cite{CFR08} for more details.

\subsection{Degenerate equations in the nondivergence setting}
\subsubsection{\bf Wellposedness}
We consider the following degenerate parabolic equation 
\begin{equation}
\left\{
\begin{aligned}
& u_{t}(x,t) -a(x) u_{xx}(x,t)= G(x,t), &&\quad \text { in } Q,\\
& u(x,t)=0, &&\quad \text{ on } \Sigma,\\
&u(x,0)=f(x), &&\;\quad x\in \Omega.
\label{eq1to1}
\end{aligned}
\right.
\end{equation}
Let us consider the natural Hilbert spaces to study \eqref{eq1to1}:
$$L^2_{\frac{1}{a}}(0,1)=\left\lbrace u 
 \in L^2(0,1)\; |\;\|u\|_{\frac{1}{a}}<\infty \right\rbrace, \; \; \|u\|_{\frac{1}{a}}^2=\int_{0}^1u^2 \frac{1}{a} dx;$$
$$H^1_{\frac{1}{a}}(0,1)=L^2_{\frac{1}{a}}(0,1) \cap H^1_{0}(0,1), \; \; \|u\|_{1,\frac{1}{a}}^2= \|u\|_{\frac{1}{a}}^2+\| u_{x}\|_{L^2(0,1)}^2;$$
$$H^2_{\frac{1}{a}}(0,1)=\left\lbrace u 
 \in H^1_{\frac{1}{a}}(0,1)\; | \; a u_{xx}\in L^2_{\frac{1}{a}}(0,1)\right\rbrace, \; \; \|u\|_{2,\frac{1}{a}}^2= \|u\|_{1,\frac{1}{a}}^2+\int_{0}^1 au_{xx}^2  dx.$$
Then, we consider the linear unbounded operator $(A,D(A))$ defined by
$$ Au=au_{xx}, \qquad u \in  D(A):= H^2_{\frac{1}{a}}(0,1).$$
Additionally, the following Green's formula holds:
\begin{lemma}
     Given $(u,v) \in H^2_{\frac{1}{a}}(0,1) \times H^1_{\frac{1}{a}}(0,1)$, we have 
     $$ \int_{0}^1 u_{xx}v dx= - \int_{0}^1 u_{x}v_{x} dx .$$
\end{lemma}
Next, we recall the semigroup generation by $(A,D(A))$. We refer to \cite{Paz83} for the terminology.
\begin{theorem}\label{thmg1}
The operator $(A,D(A))$ is self-adjoint and m-dissipative \\ in $L_{\frac{1}{a}}^2(0,1)$. It generates then an analytic $C_0$-semigroup on $L_{\frac{1}{a}}^2(0,1).$
\end{theorem}
Consequently, the following well-posedness and regularity result holds.
\begin{theorem}\label{thmgenni}
     For all $ G\in L^{2}(0,T; L^{2}_{\frac{1}{a}}(0,1))$ and $u_{0} \in L^{2}_{\frac{1}{a}}(0,1)$, there is a unique weak solution $u \in C([0,T],L^{2}_{\frac{1}{a}}(0,1))\cap  L^{2}(0,T; H^{1}_{\frac{1}{a}}(0,1) )$ of \eqref{eq1to1}. Moreover, if $f\in H^{1}_{\frac{1}{a}}(0,1) $, then 
     $$u \in H^1\left(0,T; L^2_{\frac{1}{a}}(0,1)\right)\cap L^2\left(0,T;H^{1}_{\frac{1}{a}}(0,1)\right)\cap C\left([0,T];H^{1}_{\frac{1}{a}}(0,1)\right),$$
     and there exists a positive constant $C$ (independent of $(u_0,G)$) such that
     $$ \sup_{t \in [0,T]}\| u\|^2_{L^{2}_{\frac{1}{a}}(0,1)}+\int_{0}^1\| u(t)\|^2_{H^{1}_{\frac{1}{a}}(0,1)}dt \leqslant C\left(\| u_{0}\|^2_{L^{2}_{\frac{1}{a}}(0,1)}+ \| G\|^2_{L^{2}_{\frac{1}{a}}(Q)}\right).$$
 \end{theorem}

\subsubsection{\bf Carleman estimate}\label{secCar1}
Now, we prove a new Carleman estimate that will be useful to prove the stability of the backward problem of the linear part of the viscous Hamilton-Jacobi equation \eqref{eq1to1}. We adopt the weight function in \cite{CY23} defined by
\begin{equation}
\varphi(t)=e^{\lambda t}, \;\; t>0,
\end{equation}
where $\lambda>0$ is a sufficiently large parameter. The weight function $\varphi$ is simple and independent of $x$, which makes it effective for backward problems associated with degenerate equations.

\begin{lemma}\label{lem1}
There is a constant $\lambda_{0}>0$ so that for any $\lambda>\lambda_{0},$ there exists a constant $s_{0}(\lambda)>0$ and a positive constant $C$ such that
\begin{equation}\label{Carlm HJB}
    \begin{aligned}
           &\int_{Q} \left(\frac{1}{a(x)}\vert u_{t}\vert^2+ a(x)\vert u_{xx}\vert^2+ s\lambda \varphi \vert u_{x}\vert^2+ \frac{s^2\lambda^2\varphi^2}{a(x)}\vert u\vert^2\right)e^{2s\varphi}dxdt\\&\leqslant C  \int_{Q}s\varphi\frac{1}{a(x)}\vert Ge^{s\varphi}\vert^2 dxdt +Cs\left( s\lambda\varphi(T)\| u(\cdot,T)\|_{\frac{1}{a}}^2+\| u(\cdot,T)\|_{1,\frac{1}{a}}^2 \right)e^{2s\varphi(T)}\\& +Cs\left( s\lambda\| u(\cdot,0)\|_{\frac{1}{a}}^2+\| u(\cdot,0)\|_{1,\frac{1}{a}}^2\right)e^{2s}
    \end{aligned}
\end{equation}
for all large $s\ge s_{0}(\lambda)$, where $u \in L^2\left(0,T;H^2_{\frac{1}{a}}(0,1)\right)\bigcap H^1\left(0,T; H^1_{\frac{1}{a}}(0,1)\right)$ is the solution of \eqref{eq1to1}.
\end{lemma}
\begin{proof}
Set 
$$ Lu :=u_{t} - a(x)  u_{xx}(x,t), \qquad w=e^{s\varphi}u, \qquad Pw=e^{s\varphi}L\left(e^{-s\varphi}w\right)= e^{s\varphi} G.$$
We have $$ e^{s\varphi} \left(e^{-s\varphi}w\right)_{t}=w_{t} -s\lambda \varphi w, \qquad e^{s\varphi} a(x) \left(w e^{-s\varphi}\right)_{xx}=a(x)w_{xx}.$$
Then $$ Pw= e^{s\varphi}L\left(e^{-s\varphi}w\right)= P_{-}w+P_{+}w =e^{s\varphi} G,$$
   where 
   $$P_{-}w= w_{t}, \qquad P_{+}w = -s\lambda \varphi w -a(x)w_{xx}.$$
   By taking the norm $\| \cdot\|^2_{L^2(0,T;L^2_{\frac{1}{a}}(0,1))}$ in the previous equality, we obtain
\begin{align*}
&\| e^{s\varphi}G\|^2_{L^2(0,T;L^2_{\frac{1}{a}}(0,1))}=\int_{Q} \frac{1}{a(x)}\vert w_{t}\vert^2 dx dt + 2 \int_{Q} \frac{1}{a(x)}w_{t} \left(-s\lambda \varphi w -a(x)w_{xx} \right)\\
& +\int_{Q}\frac{1}{a(x)}\left\vert - s\lambda \varphi w -a(x)w_{xx} \right\vert^2dxdt\\
& \geqslant \int_{Q}\frac{1}{a(x)} \vert w_{t}\vert^2dxdt - 2 \int_{Q}w_{t}w_{xx} dx dt+2 \int_{Q}\frac{1}{a(x)}w_{t}  (-s\lambda\varphi)wdxdt\\
&:= \int_{Q} \frac{1}{a(x)}\vert w_{t}\vert^2dxdt+ J_{1}+J_{2}.
\end{align*}
Hence 
$$ \int_{Q}\frac{1}{a(x)}\vert G\vert^2 e^{2s\varphi}dx dt\geqslant J_{1}+J_{2},$$
\begin{align}\label{2}
    \int_{Q}\frac{1}{a(x)} \vert w_{t}\vert^2dxdt \leqslant \int_{Q}\frac{1}{a(x)}\vert G\vert^2 e^{2s\varphi}dx dt -J_{1}-J_{2},
\end{align}
\begin{align*}
    \int_{Q}\frac{1}{a(x)} \vert P_{+}w\vert^2dxdt \leqslant \int_{Q}\frac{1}{a(x)}\vert G\vert^2 e^{2s\varphi}dx dt -J_{1}-J_{2}.
\end{align*}
We assume that $s>1$ and $\lambda>1.$
Using the boundary conditions and integration by parts gives
\begin{equation}\label{3}
    \begin{aligned}
    J_{1}&=\int_{Q}(w_{x}^2)_{t}dx dt\\
    &= \int_{\Omega}\left[\vert w_{x}(T,x)\vert^2- \vert w_{x}(0,x)\vert^2\right] dx\\ &\leqslant  \int_{\Omega}\left[\vert w_{x}(T,x)\vert^2+\vert w_{x}(0,x)\vert^2\right] dx.
    \end{aligned}
\end{equation}
On the other hand,
\begin{equation}\label{4}
    \begin{aligned}
    J_{2} &= -s\lambda \int_{Q} \frac{1}{a(x)}\left(w^2\right)_{t}\varphi\, dxdt\\
     &= s\lambda^2 \int_{Q}\varphi \frac{1}{a(x)} w^2 \, dx dt-s\lambda \int_{\Omega}\frac{1}{a(x)}\left(\varphi(T)\vert w(x,T)\vert^2-\vert w(x,0)\vert^2\right)dx.
\end{aligned}
\end{equation}
Then
\begin{equation}\label{5}
\begin{aligned}
    \| e^{s\varphi}G\|^2_{L^2(0,T;L^2_{\frac{1}{a}}(0,1))} &\geqslant s\lambda^2 \int_{Q}\varphi\frac{1}{a(x)} w^2 \,dxdt \\& -s\lambda \int_{\Omega} \frac{1}{a(x)}\left(\varphi(T)\vert w(x,T)\vert^2+\vert w(x,0)\vert^2\right)dx\\& \quad -\int_{\Omega}\left( \vert w_{x}(x,T)\vert^2+\vert w_{x}(x,0)\vert^2\right)\, dx.
\end{aligned}
\end{equation}
\begin{equation*}\label{6}
    \begin{aligned}
    \int_{Q}\frac{1}{a(x)}\left(Pw\right)w\, dxdt&=\int_{Q} \frac{1}{a(x)}w_{t}w\, dx dt - \int_{Q}s\lambda \varphi\frac{1}{a(x)} w^2\, dx dt\\&- \int_{Q} \frac{1}{a(x)} \left( a(x)w_{xx}\right)w \,dxdt\\
    &= \int_{Q} \frac{1}{a(x)}w_{t}w\, dxdt -\int_{Q}s\lambda \varphi\frac{1}{a(x)} w^2\, dxdt- \int_{Q} w_{xx}w \,dxdt\\
    &  =I_{1}+I_{2}+I_{3}.
\end{aligned}
\end{equation*}
\begin{align*}
   \left\vert I_{1}\right\vert &= \left\vert \frac{1}{2} \int_{Q}\frac{1}{a(x)}(w^2)_{t} \, dxdt\right\vert\\ & =\frac{1}{2}\left\vert \int_{\Omega}\frac{1}{a(x)}\left[ \vert w(x,t)\vert^2\right]_{t=0}^{t=T}\, dx\right\vert \\&\leqslant \frac{1}{2}\int_{\Omega}\frac{1}{a(x)}\left(\vert w(x,T)\vert^2+\vert w(x,0)\vert^2 \right)\, dx.
\end{align*}
Next
$$ \vert I_{2}\vert = \left\vert -\int_{Q} s\lambda \varphi \frac{1}{a(x)} w^2\, dxdt\right\vert \leq  \int_{Q} s\lambda \varphi \frac{1}{a(x)} w^2 \, dxdt,$$
and 
\begin{align*}
     I_{3}&= -\int_{Q}w_{xx}w dxdt= \int_{Q}\vert w_{x}\vert^2  dxdt.
\end{align*}
Therefore
\begin{align*}
    \int_{Q}\lambda \frac{1}{a(x)}\left(Pw \right)w dxdt &\geqslant   \int_{Q}\lambda \vert w_{x}\vert^2 dx dt- \int_{Q} s\lambda^2 \varphi \frac{1}{a(x)} w^2 dxdt\\& - \int_{\Omega}\frac{\lambda}{2a(x)}\left(\vert w(x,T)\vert^2+\vert w(x,0)\vert^2 \right)\, dx.
\end{align*}
Moreover
\begin{align*}
\left\vert \int_{Q} \frac{1}{a(x)} \lambda(Pw)wdxdt\right\vert  & \leqslant \frac{1}{2}\int_{Q}\frac{1}{a(x)}\vert Pw\vert^2 dx dx dt+\frac{\lambda^2}{2}\int_{Q}\frac{1}{a(x)}\vert w\vert^2 dx dt\\ & =\frac{1}{2} \int_{Q} \frac{1}{a(x)}\vert G \vert^2 e^{2s\varphi} dxdt +\frac{\lambda^2}{2}\int_{Q}\frac{1}{a(x)}\vert w\vert^2 dx dt.
\end{align*}
\begin{align*}
    \lambda \int_{Q} \vert w_{x}\vert^2  dx dt &\leqslant \int_{Q}s\lambda^2\varphi \frac{1}{a(x)}w^2 dxdt + \frac{1}{2} \int_{Q} \frac{1}{a(x)}\vert G \vert^2 e^{2s\varphi} dxdt \\&+\frac{\lambda^2}{2}\int_{Q}\frac{1}{a(x)}\vert w\vert^2 dx dt  + \frac{\lambda}{2}\int_{\Omega}\frac{1}{a(x)}\left(\vert w(x,T)\vert^2+\vert w(x,0)\vert^2 \right)\, dx.
\end{align*}
We use \eqref{5} to estimate the first term on the right-hand side. Then
\begin{equation}\label{7}
\begin{aligned}
 & \lambda \int_{Q} \vert w_{x}\vert^2  dx dt  \leqslant C_{3}\int_{Q}\frac{1}{a(x)}\vert Ge^{s\varphi}\vert^2 dxdt + C_{3}\int_{Q}\lambda^2 \frac{1}{a(x)} w^2 dxdt  \\ &+C_{3}s\lambda \left(\varphi(T)\| w(\cdot,T)\|_{\frac{1}{a}}^2+\| w(\cdot,0)\|_{\frac{1}{a}}^2 \right)+C_{3}\left( \| w_{x}(\cdot,T)\|_{L^2(0,1)}^2+\| w_{x}(\cdot,0)\|_{L^2(0,1)}^2\right).
\end{aligned}
\end{equation}
Adding \eqref{5} and \eqref{7}, we obtain
\begin{align*}
    &\int_{Q} s\lambda^2 \varphi \frac{1}{a(x)} w^2 dxdt+\lambda \int_{Q} \vert w_{x}\vert^2  dx dt \\& \leqslant C_{4}  \int_{Q} \frac{1}{a(x)}\vert Ge^{s\varphi}\vert^2 dxdt +C_{4}\int_{Q}\lambda^2  \frac{1}{a(x)} w^2 dxdt\\
    &\quad +C_{4}s\lambda \left(\varphi(T)\| w(\cdot,T)\|_{\frac{1}{a}}^2+\| w(\cdot,0)\|_{\frac{1}{a}}^2 \right) \\&+C_{4}\left( \| w_{x}(\cdot,T)\|_{L^2(0,1)}^2+\| w_{x}(\cdot,0)\|_{L^2(0,1)}^2\right).
\end{align*}
Given that $\varphi(t)=e^{\lambda t}\geqslant 1$, by choosing $s>0$ sufficiently large, the second term on the right-hand side can be absorbed into the left-hand side. Then,
\begin{equation}\label{8}
\begin{aligned}
    &\int_{Q} s\lambda^2 \varphi \frac{1}{a(x)}w^2 dxdt+\lambda \int_{Q}\vert w_{x}\vert^2  dx dt \leqslant C_{4}  \int_{Q}\frac{1}{a(x)}\vert Ge^{s\varphi}\vert^2 dxdt \\&+C_{4}s\lambda \left(\varphi(T)\| w(\cdot,T)\|_{\frac{1}{a}}^2+\| w(\cdot,0)\|_{\frac{1}{a}}^2 \right)\\& +C_{4}\left( \| w_{x}(\cdot,T)\|_{L^2(\Omega)}^2+\| w_{x}(\cdot,0)\|_{L^2(\Omega)}^2\right).
\end{aligned}
\end{equation}
We now focus on estimating $\vert w_{t}\vert^2$. Since we have $u=e^{-s\varphi}w,$ $u_{t}=-s\lambda \varphi e^{-s\varphi}w+e^{-s\varphi}w_{t},$ and 
$$ \frac{1}{s\varphi a(x)}\vert u_{t}\vert^2e^{2s\varphi }\leqslant 2s\lambda^2 \varphi\frac{1}{a(x)} w^2+\frac{2}{s\varphi a(x)}\vert w_{t}\vert^2. $$
Let $\epsilon \in \left(0,\frac{1}{2}\right)$ be an arbitrary parameter to be chosen. We have $\frac{1}{s\varphi}= \frac{1}{se^{\lambda t}}\leq \frac{1}{s}\leq \frac{1}{2}$ for $s \geqslant 2.$ Hence, for all $s>0$ and $\lambda> 0$ large enough, it follows by \eqref{2} that 
\begin{equation*}
    \begin{aligned}
        &\int_{Q}\frac{\epsilon}{s\varphi a(x)}\vert u_{t}\vert^2 e^{2s\varphi}\, dxdt\leqslant 2\epsilon\int_{Q} s \lambda^2 \varphi \frac{1}{a(x)}w^2+2\epsilon\int_{Q}\frac{1}{s\varphi a(x)}\vert w_{t}\vert^2dxdt\\ &\leqslant 2\epsilon \int_{Q}s\lambda^2\varphi \frac{1}{a(x)} w^2 dxdt+\epsilon \int_{Q}\frac{1}{a(x)}\vert w_{t}\vert^2dxdt\\ & \leqslant 2\epsilon \int_{Q} s\lambda^2\varphi \frac{1}{a(x)} w^2 dxdt + \epsilon \int_{Q} \frac{1}{a(x)} \vert G\vert^2 e^{2s\varphi}dxdt + \epsilon \left(-J_{1}-J_{2}\right).
    \end{aligned}
\end{equation*}
By applying \eqref{3} and \eqref{4}, we obtain
\begin{equation}\label{10}
    \begin{aligned}
        &\int_{Q}\frac{\epsilon}{s\varphi a(x)}\vert u_{t}\vert^2 e^{2s\varphi}\, dxdt\leqslant \epsilon \int_{Q} s\lambda^2\varphi\frac{1}{a(x)} w^2 dxdt + \epsilon \int_{Q} \frac{1}{a(x)}\vert G\vert^2 e^{2s\varphi}dxdt \\ & +\int_{\Omega}\epsilon\left(\vert  w_{x}(x,T)\vert^2+ \vert   w_{x}(x,0)\vert^2\right)dx+\epsilon s\lambda \int_{\Omega}\frac{1}{a(x)}\left(\varphi(T)\vert w(x,T)\vert^2+\vert w(x,0)\vert^2\right)dx.
    \end{aligned}
\end{equation}
Adding \eqref{8} to \eqref{10} yields 
\begin{equation}\label{11}
    \begin{split}
         &\int_{Q} s\lambda^2 \varphi \frac{1}{a(x)}w^2 dxdt+ \int_{Q} \lambda \vert w_{x}\vert^2 dx dt+\int_{Q}\frac{\epsilon}{s\varphi a(x)}\vert u_{t}\vert^2 e^{2s\varphi}\, dxdt\\& \leqslant C_{5}  \int_{Q}\frac{1}{a(x)}\vert Ge^{s\varphi}\vert^2 dxdt + 2\epsilon \int_{Q} s\lambda^2\varphi \frac{1}{a(x)}w^2 dxdt \\&+C_{5}s\lambda \left(\varphi(T)\| w(\cdot,T)\|_{\frac{1}{a}}^2+\| w(\cdot,0)\|_{\frac{1}{a}}^2 \right)\\&\quad +C_{5}\left( \| w_{x}(\cdot,T)\|_{L^2(0,1)}^2+\| w_{x}(\cdot,0)\|_{L^2(0,1)}^2\right).
    \end{split}
\end{equation}
Next, we will estimate $ a(x) w_{xx}^2$. We have $a(x) w_{xx}=-P_{+}w-s\lambda \varphi w.$ Then 
\begin{align*}
    \frac{1}{s\varphi}a(x)\vert w_{xx}\vert^2 & \leqslant 2 \frac{1}{s\varphi a(x)}\vert P_{+}w\vert^2 +  2\frac{s \lambda^2 \varphi}{a(x)}\vert w\vert^2 \\ &
   \leqslant \frac{1}{a(x)} \vert P_{+}w \vert^2 e^{2s\varphi}  + 2\frac{s \lambda^2 \varphi}{a(x)}\vert w\vert^2 .
    \end{align*}
Therefore, for all large $s>0$ and $\lambda> 0, $ we have
    \begin{equation}\label{12}
        \begin{split}
            \int_{Q} \epsilon\frac{1}{s\varphi}a(x)\vert w_{xx}\vert^2 dx dt &\leqslant \int_{Q}\epsilon \frac{1}{a(x)} \vert G\vert^2 e^{2s\varphi}dxdt -\epsilon(J_{1}+J_{2}) + \int_{Q}2\epsilon  \frac{s \lambda^2 \varphi}{a(x)}\vert w\vert^2 dx dt\\&
            \leq \int_{Q}\epsilon \frac{1}{a(x)} \vert G\vert^2 e^{2s\varphi}dxdt +\int_{Q}\epsilon  \frac{s \lambda^2 \varphi}{a(x)}\vert w\vert^2 dx dt\\&+C\int_{0}^1\epsilon\left(\vert  w_{x}(x,T)\vert^2+ \vert  w_{x}(x,0)\vert^2\right)dx
            \\&+\epsilon s\lambda \int_{0}^1\frac{1}{a(x)}\left(\varphi(T)\vert w(x,T)\vert^2+\vert w(x,0)\vert^2\right)dx.
        \end{split}
    \end{equation}
    Adding \eqref{11} to \eqref{12}, we obtain 

\begin{equation*}
    \begin{split}
          &\int_{Q} s\lambda^2 \varphi \frac{1}{a(x)}w^2 dxdt + \int_{Q} \epsilon\frac{1}{s\varphi}a(x)\vert w_{xx}\vert^2 dx dt\\&+\int_{Q} \lambda \vert w_{x}\vert^2 dx dt+\int_{Q}\frac{\epsilon}{s\varphi a(x)}\vert u_{t}\vert^2 e^{2s\varphi}\, dxdt\\& \leqslant C_{5}  \int_{Q}\frac{1}{a(x)}\vert Ge^{s\varphi}\vert^2 dxdt + 4\epsilon \int_{Q} s\lambda^2\varphi \frac{1}{a(x)}w^2 dxdt \\& \quad+C_{5}s\lambda \left(\varphi(T)\| w(\cdot,T)\|_{\frac{1}{a}}^2+\| w(\cdot,0)\|_{\frac{1}{a}}^2 \right)\\&\quad+C_{5}\left( \| w_{x}(\cdot,T)\|_{L^2(\Omega)}^2+\| w_{x}(\cdot,0)\|_{L^2(\Omega)}^2\right).
    \end{split}
\end{equation*}
We set $v=\varphi^{\frac{1}{2}}u$. Then $ v_{t}=\frac{1}{2}\lambda \varphi^{\frac{1}{2}}u+\varphi^{\frac{1}{2}}u_{t}$, and so
\begin{align}
    v_{t}-a(x)v_{xx}=\varphi^{\frac{1}{2}}(u_{t}-a(x)u_{xx})+ \frac{1}{2}\lambda\varphi^{\frac{1}{2}} u= \varphi^{\frac{1}{2}} G +\frac{1}{2}\lambda\varphi^{\frac{1}{2}} u.
\end{align}
Therefore, we obtain 
\begin{equation*}\label{14}
    \begin{split}
         & \int_{Q} \frac{1}{\varphi}\left(a(x)\vert v_{xx}\vert^2+\frac{1}{ a(x)}\vert v_{t}\vert^2 \right)e^{2s\varphi}dx dt\\&+\int_{Q} \lambda s\vert v_{x}\vert^2 e^{2s\varphi} dx dt+\int_{Q} s^2\lambda^2 \varphi \frac{1}{a(x)}v^2 e^{2s\varphi} dxdt\\& \leqslant C  \int_{Q}s\varphi\frac{1}{a(x)}\vert Ge^{s\varphi}\vert^2 dxdt +C\frac{1}{2}\int_{Q}\frac{s\lambda^2\varphi}{a(x)}u^2 e^{2s\varphi}dxdt  \\ &\quad +C s\left( s\lambda\varphi(T)\| v(\cdot,T)\|_{\frac{1}{a}}^2+\| v(\cdot,T)\|_{1,\frac{1}{a}}^2 \right)e^{2s\varphi(T)}\\& 
         \quad +Cs\left( s\lambda\| v(\cdot,0)\|_{\frac{1}{a}}^2+\| v(\cdot,0)\|_{1,\frac{1}{a}}^2\right)e^{2s}.
    \end{split}
\end{equation*}
On the other hand,
\begin{equation*}
    \begin{aligned}
        &\varphi\vert u_{t}\vert^2 \leq 2\vert v_{t}\vert^2 +\lambda^2\varphi\vert u\vert^2,\qquad 
        \varphi \vert u_{xx}\vert^2 = \vert v_{xx}\vert^2.
    \end{aligned}
\end{equation*}
Hence, 
\begin{align*}
&\frac{1}{a(x)}\vert u_{t}\vert^2+ a(x)\vert u_{xx}\vert^2 \leqslant \frac{1}{\varphi}\left( \frac{2}{a(x)}\vert v_{t}\vert^2 +a(x)\vert v_{xx}\vert^2\right)+ \frac{\lambda^2}{a(x)}\vert u \vert^2.
\end{align*}
Therefore,
\begin{align*}
    &\frac{1}{a(x)}\vert u_{t}\vert^2+ a(x)\vert u_{xx}\vert^2+ s\lambda \varphi \vert u_{x}\vert^2+ \frac{s^2\lambda^2\varphi^2}{a(x)}\vert u\vert^2\\&
    \leqslant \frac{1}{\varphi}\left( \frac{2}{a(x)}\vert v_{t}\vert^2 +a(x)\vert v_{xx}\vert^2\right)+ \frac{\lambda^2}{a(x)}\vert u \vert^2 + s\lambda\varphi \vert u_{x}\vert^2+ \frac{s^2\lambda^2\varphi^2}{a(x)}\vert u\vert^2 \\ &\leqslant C\frac{1}{\varphi}\left( \frac{1}{a(x)}\vert v_{t}\vert^2 +a(x)\vert v_{xx}\vert^2\right)+ s\lambda \vert v_{x}\vert^2+ \frac{s^2\lambda^2\varphi}{a(x)}\vert v\vert^2+ \frac{\lambda^2}{a(x)}\vert u \vert^2 . 
\end{align*}
Applying \eqref{1}, we obtain
\begin{equation*}
    \begin{split}
        &\int_{Q} \left(\frac{1}{a(x)}\vert u_{t}\vert^2+ a(x)\vert u_{xx}\vert^2+ s\lambda \varphi \vert u_{x}\vert^2+ \frac{s^2\lambda^2\varphi^2}{a(x)}\vert u\vert^2\right)e^{2s\varphi}dxdt\\&
        \leqslant C\int_{Q}\frac{1}{\varphi}\left( \frac{1}{a(x)}\vert v_{t}\vert^2 +a(x)\vert v_{xx}\vert^2\right)e^{2s\varphi}dxdt+\int_{Q}s\lambda \vert v_{x}\vert^2e^{2s\varphi}\\ &\quad +\int_{Q} \frac{s^2\lambda^2\varphi}{a(x)}\vert v\vert^2 e^{2s\varphi} dxdt+ \int_{Q} \frac{\lambda^2}{a(x)}\vert u \vert^2 e^{2s\varphi} dxdt\\&\leqslant C  \int_{Q}s\varphi\frac{1}{a(x)}\vert Ge^{s\varphi}\vert^2 dxdt +C\frac{1}{2}\int_{Q}\frac{s\lambda^2\varphi}{a(x)}u^2 e^{2s\varphi}dxdt \\&\quad +C s\left( s\lambda\varphi(T)\| v(\cdot,T)\|_{\frac{1}{a}}^2+\| v(\cdot,T)\|_{1,\frac{1}{a}}^2 \right)e^{2s\varphi(T)}\\& \quad +Cs\left( s\lambda\| v(\cdot,0)\|_{\frac{1}{a}}^2+\| v(\cdot,0)\|_{1,\frac{1}{a}}^2\right)e^{2s}+\int_{Q} \frac{\lambda^2}{a(x)}\vert u \vert^2 e^{2s\varphi} dxdt.
    \end{split}
\end{equation*}
For sufficiently large values of  $s>0$, we can absorb the second term and the last term to the right-hand side, and we obtain
\begin{align*}
    &\int_{Q} \left(\frac{1}{a(x)}\vert u_{t}\vert^2+ a(x)\vert u_{xx}\vert^2+ s\lambda \varphi \vert u_{x}\vert^2+ \frac{s^2\lambda^2\varphi^2}{a(x)}\vert u\vert^2\right)e^{2s\varphi}dxdt\\&\leqslant C  \int_{Q}s\varphi\frac{1}{a(x)}\vert Ge^{s\varphi}\vert^2 dxdt +C s\left( s\lambda\varphi(T)\| v(\cdot,T)\|_{\frac{1}{a}}^2+\| v(\cdot,T)\|_{1,\frac{1}{a}}^2 \right)e^{2s\varphi(T)}\\&\quad +Cs\left( s\lambda\| v(\cdot,0)\|_{\frac{1}{a}}^2+\| v(\cdot,0)\|_{1,\frac{1}{a}}^2\right)e^{2s}.
\end{align*}
\end{proof}

\subsubsection{\textbf{Stability of the backward problem}}
Now, we prove H\"older stability for past states $u(\cdot,t_0)$ when $0<t_{0}<T$. Note that this idea was first used in \cite{CY23} for degenerate equations in divergence form.
\begin{theorem}\label{15}
Let $u \in L^2\left(0,T;H^2_{\frac{1}{a}}(0,1)\right)\bigcap H^1\left(0,T; H^1_{\frac{1}{a}}(0,1)\right)$ satisfy \eqref{eq1to1} and assume that 
$$\| u(\cdot,0)\|_{1,\frac{1}{a}}\leq M_1$$
for an arbitrarily fixed constant \( M_1 > 0 \). There are positive constants $C>0$ and $\theta \in (0,1)$ depending on $t_{0}$ and $M_1$ satisfying
\begin{align}
    \| u(\cdot,t_{0})\|_{\frac{1}{a}}\leqslant C\left( \| u(\cdot,T)\|^{\theta}_{1,\frac{1}{a}}+\| u(\cdot,T)\|_{1,\frac{1}{a}}\right).
\end{align}
  
\end{theorem}
\begin{proof}
Using the Carleman estimate \eqref{Carlm HJB} and the fact that $\varphi(t_{0})\leq \varphi(t)$, for $t_{0}\leq t\leq T$, yields
    \begin{align*}
        &e^{2s\varphi(t_{0})}\int_{(0,1)\times (t_{0},T)} \left(\frac{1}{a(x)}\vert u_{t}\vert ^2 +s^2\lambda^2 \varphi^2 \frac{1}{a(x)}u^2\right)dx dt \\&\leqslant Cs^2\lambda \varphi(T)\| u(\cdot,T)\|_{1,\frac{1}{a}}^2 e^{2s\varphi(T)}+Cs^2\lambda M_1^2 e^{2s}.
    \end{align*}
For $\lambda>0$ sufficiently large, we have
 \begin{align*}
        &\int_{(0,1)\times (t_{0},T)} \left(\frac{1}{a(x)}\vert u_{t}\vert ^2 +s^2\lambda^2 \varphi^2 \frac{1}{a(x)}u^2\right)dx dt \\&\leqslant Cs^2\| u(\cdot,T)\|_{1,\frac{1}{a}}^2e^{2s(\varphi(T)-\varphi(t_{0}))} +Cs^2 M_1^2 e^{-2s\alpha(t_{0})},
    \end{align*}
    with $\alpha(t_{0}):=\varphi(t_{0})-1$ and the constant $C$ depends on $T$ and $\lambda$.  Hence,  
    \begin{align}\label{16}
        \| u_{t}\|_{L^{2}(t_{0},T;L^2_{\frac{1}{a}}(0,1))}^2 \leqslant Cs^2D_{0}^2e^{2s(\varphi(T)-\varphi(t_{0})}+Cs^2 M_1^2 e^{-2s\alpha(t_{0})},
    \end{align}
    where $D_{0}:=\| u(\cdot,T)\|_{1,\frac{1}{a}}.$
    Since we have 
    $$ u(x,t_{0})=\int_{T}^{t_{0}}u_{t}(x,t)dt +u(x,T)\qquad  \forall x \in (0,1),$$
there exists a constant $C_{6}>0$ such that for all $t_{0}\in (0,T)$ we have
\begin{align}
    \| u(\cdot,t_{0})\|_{\frac{1}{a}}^2\leqslant C_{6} \| u_{t}\|_{L^{2}(t_{0},T; L^2_{\frac{1}{a}}(0,1))}^2+ C_{6} \| u(\cdot,T)\|^2_{\frac{1}{a}}.
\end{align}
Using \eqref{16} to estimate the first term on the right-hand side, along with the fact that $\varphi(T)>1,$ we obtain
\begin{equation*}
    \begin{aligned}
        \| u(\cdot,t_{0})\|_{\frac{1}{a}}^2&\leqslant C_{6}\left( Cs^2D_{0}^2e^{2s(\varphi(T)-\varphi(t_{0})}+Cs^2 M_1^2 e^{-2s\alpha(t_{0})}\right)+  C_{6} \| u(\cdot,T)\|^2_{\frac{1}{a}}  \\&\leqslant 
C_{7}s^2D_{0}^2e^{2s\varphi(T)}+C_{7}s^2M_1^2e^{-2s\alpha(t_{0})} \\ & \leqslant C_{7}D_{0}^2e^{3s\varphi(T)} +C_{7}M_1^2e^{-s\alpha(t_{0})}
        \end{aligned}
\end{equation*}
for all $s>s_{0}$ and all $t_{0}\in (0,T)$, and this yields 
\begin{align}\label{18}
   \| u(\cdot,t_{0})\|_{\frac{1}{a}}^2 \leqslant C_{8}D_{0}^2e^{3s\varphi(T)} +C_{8}M_1^2e^{-s\alpha(t_{0})} \;\;\; \text{for all } \;\; s>0,
\end{align}
with $C_{8}=C_{7}e^{6s_{0}\varphi(T)}.$
To obtain the desired result, we study two possible cases.
\begin{itemize}
    \item If $M_1\leqslant D_{0}$,  taking $s=0$, the inequality \eqref{18} yields 
    $$ \| u(\cdot,t_{0})\|_{\frac{1}{a}}^2 \leq 2 C_{8}D_{0}^2.$$
    \item  If $M_1>D_{0}$, we choose $s>0$ such that 
    $$ D_{0}^2e^{3s\varphi(T)}=M_1^2e^{-s\alpha(t_{0})},
    $$
which yields
$$ s= \frac{2}{3\varphi(T)+\alpha(t_{0})}\log\frac{M_1}{D_{0}}>0.
$$
Then, 
$$ \| u(\cdot,t_{0})\|_{\frac{1}{a}}^2 \leqslant 2 C_{8}M_1^{\frac{6\varphi(T)}{3\varphi(T)+\alpha(t_{0})}}D_{0}^{\frac{2\alpha(t_{0})}{3\varphi(T)+\alpha(t_{0})}}.$$
Thus, taking $\theta:=\frac{\alpha(t_{0})}{3\varphi(T)+\alpha(t_{0})}\in (0,1)$ concludes the proof.
\end{itemize}
\end{proof}

Next, we prove logarithmic stability for initial states, i.e., when $t_{0}=0$.
\begin{theorem}\label{thmstab}
Let $u,u_t,u_{tt} \in L^2\left(0,T;H^2_{\frac{1}{a}}(0,1)\right)\bigcap H^1\left(0,T; H^1_{\frac{1}{a}}(0,1)\right)$ satisfy \eqref{eq1to1} such that
   $$\| u_{tt}(\cdot,0)\|_{1,\frac{1}{a}}\leqslant M_{1}$$
for an arbitrarily fixed constant \( M_1 > 0 \) and any \( 0 < \alpha < 1 \), there is a constant \( C > 0 \) such that
   \begin{align*}
       \| u(\cdot,0)\| _{\frac{1}{a}}^2\leqslant C \left(\log\frac{1}{D}\right)^{-\alpha},
   \end{align*}
where $D:=\displaystyle \sum_{k=0}^2\| \partial_{t}^k u(\cdot,T)\|_{1,\frac{1}{a}}$ is sufficiently small.
\end{theorem}

\begin{proof}
The proof is based on the proof of the previous theorem (Case $0<t_{0}<T$) and 
$$u(x,0)=\int_{T}^0u_{t}(x,t)dt +u(x,T), \;\; x \in (0,1). $$
    Setting $z=u_{tt}$, we have
\begin{equation}
    \left\{
\begin{aligned}
& z_{t}(x,t) - a(x) z_{xx}(x,t)= G_{tt}(x,t) &&\quad\text{ in } Q,\\
& z(x,t)=0 &&\quad \text{ on } \Sigma,\\
& z(x,0)=z_0(x)  &&\quad \text{ in } \Omega.
\label{eq1to2}
\end{aligned}
\right.
\end{equation}
Using the Carleman estimate \eqref{Carlm HJB} and the fact that $\varphi(t_{0})\leq \varphi(t)$ for $t_{0}\leq t\leq T$, we obtain 
    \begin{align*}
        &e^{2s\varphi(t_{0})}\int_{(0,1)\times (t_{0},T)} \left(\frac{1}{a(x)}\vert z_{t}\vert ^2 +s^2\lambda^2 \varphi^2 \frac{1}{a(x)}z^2\right)dx dt \\&\leqslant Cs^2\lambda \varphi(T)\| z(\cdot,T)\|_{1,\frac{1}{a}}^2 e^{2s\varphi(T)}+Cs^2 \lambda \| z(\cdot,0)\|_{1,\frac{1}{a}}^2 e^{2s},
    \end{align*}
which implies for $ \lambda >0$ sufficiently large and $s\geqslant s_{0}$
\begin{align*}
     \int_{(0,1)\times (t_{0},T)} \left(\frac{1}{a(x)}\vert u_{ttt}\vert ^2 +s^2\lambda^2 \varphi^2 \frac{1}{a(x)}\vert u_{tt} \vert^2\right)dx dt \leqslant & Cs^2\| u_{tt}(\cdot,T)\|_{1,\frac{1}{a}}^2 e^{2s(\varphi(T)-\varphi(t_{0})}\\ &+Cs^2 \| u_{tt}(\cdot,0)\|_{1,\frac{1}{a}}^2 e^{-2s\alpha(t_{0})},
    \end{align*}
    with $\alpha(t_{0}):=\varphi(t_{0})-1$ and the constant $C$ depends on $T$ and $\lambda$. Hence, 
    \begin{align}\label{19}
        \| u_{tt}\|_{L^2(t_{0},T,L^2_{\frac{1}{a}}(0,1))}\leqslant C \| u_{tt}(\cdot,T)\|_{1,\frac{1}{a}}^2e^{2s\varphi(T)}+C \| u_{tt}(\cdot,0)\|_{1,\frac{1}{a}}^2 e^{-2s\alpha(t_{0})}
    \end{align}
    for all $0<t_{0}<T.$ Since
    $$ u_{t}(x,t_{0})=\int_{T}^{t_{0}}u_{tt}(x,t)dt+u_{t}(x,T), \;\;\; x \in (0,1),$$
    we have 
    \begin{align*}
        \| u_{t}(\cdot,t_{0})\|_{\frac{1}{a}}^2\leqslant C \| u_{tt}\|_{L^2(t_{0},T,L^2_{\frac{1}{a}}(0,1))}^2+C\| u_{t}(\cdot,T)\|_{\frac{1}{a}}^2.
    \end{align*}
Using \eqref{19} to estimate the first term on the right-hand side, one obtains
    \begin{equation}\label{20}
        \begin{aligned}
             \| u_{t}(\cdot,t_{0})\|_{\frac{1}{a}}^2\leqslant & C \| u_{tt}(\cdot,T)\|_{1,\frac{1}{a}}^2e^{2s\varphi(T)}+C\| u_{tt}(\cdot,0)\|_{1,\frac{1}{a}}^2 e^{-2s\alpha(t_{0})}+C\| u_{t}(\cdot,T)\|_{\frac{1}{a}}^2\\& \leqslant C_{1} D^2e^{2s\varphi(T)}+C_{1}M_{1}^2e^{-2s\alpha(t_{0})}
        \end{aligned}
    \end{equation}
    for all $s \geqslant s_{0}$ and $0<t_{0}<T.$  Here, we used $\| u_{t}(\cdot,T)\|_{\frac{1}{a}}^2\leqslant D^2 \leqslant CD^2e^{2s\varphi(T)}$ and $\| u_{tt}(\cdot,T)\|_{1,\frac{1}{a}}^2 \leqslant D^2$.
    Since 
    $$ u(x,0)=\int_{T}^0u_{t}(x,t_{0})dt_{0}+u(x,T), \;\; x \in (0,1),$$
    we obtain 
    \begin{align*}
        \int_{0}^1 \frac{1}{a(x)}\vert u(x,0)\vert^2dx\leqslant &  2 \int_{0}^1 \frac{1}{a(x)}\left\vert \int_{T}^0u_{t}(x,t_{0})dt_{0}\right\vert^2 dx +2 \int_{0}^1 \frac{1}{a(x)}\vert u(x,T)\vert^2dx \\ &\leqslant C \int_{0}^{T} \| u_{t}(\cdot,t_{0})\|_{\frac{1}{a}}^2dt_{0}+ C \| u(\cdot,T)\|_{\frac{1}{a}}^2.
    \end{align*}
Using \eqref{20} to estimate the first term on the right-hand side, we obtain
    $$ \| u(\cdot,0)\|_{\frac{1}{a}}^2\leqslant C_{1}D^2\int_{0}^T e^{2s\varphi(T)}dt_{0}+C_{1}M_{1}^2 \int_{0}^T e^{-2s\alpha(t_{0})}dt_{0}+C\| u(\cdot,T)\|_{\frac{1}{a}}^2.$$
    Since $\| u(\cdot,T)\|_{\frac{1}{a}}^2\leqslant D^2,$ then 
    \begin{align}\label{21}
        \| u(\cdot,0)\|_{\frac{1}{a}}^2\leqslant C_{1}D^2e^{C_{2}s}+ C_{1}M_{1}^2 \int_{0}^T e^{-2s\alpha(t_{0})}dt_{0}.
    \end{align}
Using the change of variable $\tau = \alpha(t_{0})= e^{\lambda t_{0}}-1$, we have
    $$ \int_{0}^T e^{-2s\alpha(t_{0})}dt_{0}=\frac{1}{\lambda} \int_{0}^Te^{-2s\tau}\frac{1}{1+\tau}d\tau \leqslant \frac{1}{\lambda}\left[\frac{e^{-2s\tau}}{2s}\right]_{\tau =e^{\lambda T}-1}^{\tau=0}\leqslant \frac{1}{2s\lambda}.$$
    Hence, by \eqref{21}, we obtain
    \begin{align}\label{22}
        \| u(\cdot,0)\|_{\frac{1}{a}}^2\leqslant C_{3}D^2e^{C_{2}s}+ \frac{C_{3}}{s}M_{1}^2
    \end{align}
for $ s\geqslant s_{0}.$
Setting $C'=C_{3}e^{C_{2}s_{0}}$,  we obtain \eqref{22} for all $s>0$ by replacing $s:=s+s_{0}$. 
Supposing $D<1$ and  setting  
 $s=\left(\log\frac{1}{D}\right)^{\alpha}>0$,
with $0<\alpha<1$, we have
\begin{align*}
    e^{C_{2}s}D^2=\exp\left(-2\left(\log\frac{1}{D}\right)+ C_{2}\left(\log\frac{1}{D}\right)^{\alpha}\right).
\end{align*}
Moreover, there exists a constant $C_{4}>0$ such that 
$$ e^{-2\xi+C_{2}\xi^{\alpha}}\leqslant \frac{C_{4}}{\xi^{\alpha}}\;\;\; \text{for all}\;\; \xi >0.
$$
 Then,  we have 
 $$e^{C_{2}s}D^2 \leqslant C_{4}\left(\log\frac{1}{D}\right)^{-\alpha},
 $$
 and 
 $$C_{3}D^2e^{C_{2}s}+ \frac{C_{3}}{s}M_{1}^2 \leqslant C'C_{4}\left(\log\frac{1}{D}\right)^{-\alpha}+C_{3}M_{1}^2\left(\log\frac{1}{D}\right)^{-\alpha}.$$
  Thus,  we have the result.
\end{proof}

\section{Degenerate viscous Hamilton–Jacobi equation}\label{sec3}
Now, we arrive at the main result for a degenerate viscous Hamilton–Jacobi equation with a general Hamiltonian and lower-order terms. Consider the following nonlinear degenerate equation 
\begin{empheq}[left = \empheqlbrace]{alignat=2}
\begin{aligned}
& u_{t}(x,t) - a(x) u_{xx}(x,t)+ \dfrac{b(x,t)}{q}\vert u_{x}(x,t)\vert^q + d(x,t)u_{x}(x,t) +c(x,t) u(x,t)= G(x,t),\;\;\; \text{ in } Q,\\
&u(x,t)=0,\hspace{6.5cm} (x,t)\in \Sigma,
\label{23}
\end{aligned}
\end{empheq}
where $b, c, d\in L^{\infty}(Q)$, $q \geqslant 1$ and the degenerate coefficient $a(x)$ satisfies \textbf{Assumption I}.

\begin{remark}
It should be emphasized that we consider conditional stability for \eqref{23} with a general Hamiltonian in terms of $q\geqslant 1$, whereas similar works only consider a quadratic Hamiltonian with $q=2$; see e.g., \cite{Kl0} for a mean field game system.
\end{remark}

Next, we state a result of H\"older stability when $0<t_{0}<T$.
\begin{theorem}\label{stabHJ} Let $u,v  \in L^2\left(0,T;H^2_{\frac{1}{a}}(0,1)\right)\bigcap H^1\left(0,T; H^1_{\frac{1}{a}}(0,1)\right)$ satisfy \eqref{23} and 
    $$\| u(\cdot,0)\|_{1,\frac{1}{a}}\leqslant M_1, \;\;\;\| v(\cdot,0)\|_{1,\frac{1}{a}}\leqslant M_1$$
for an arbitrarily fixed constant \( M_1 > 0 \). We assume one of the following assumptions:

{\bf Case 1:}  there exist constants $C, K>0$ such that
$$\vert b(x,t)\vert +\vert d(x,t)\vert\leqslant C \sqrt{a(x)}, \quad \| u_{x}\|_{L^{\infty}(Q)}\leqslant K, \;\;\; \| v_{x}\|_{L^{\infty}(Q)}\leqslant K;$$

{\bf Case 2:} $b \in L^{\infty}(Q)$  and there exists a constant $K>0$ such that
$$\left\Vert \dfrac{\vert u_{x}\vert^{q-1}}{\sqrt{a}}\right\Vert_{L^{\infty}(Q)}\leqslant K, \qquad \left\Vert \dfrac{\vert v_{x}\vert^{q-1}}{\sqrt{a}}\right\Vert_{L^{\infty}(Q)}\leqslant K,\qquad \vert d(x,t)\vert\leqslant C \sqrt{a(x)}.$$
Then there exists a constant $C>0$ and $\theta \in (0,1)$ depending on $t_{0}$ and $M_1$ satisfying
\begin{align*}
    \| u(\cdot,t_{0})-v(\cdot,t_{0})\|_{\frac{1}{a}}\leqslant C\left( \| u(\cdot,T)-v(\cdot,T)\|^{\theta}_{1,\frac{1}{a}}+\| u(\cdot,T)-v(\cdot,T)\|_{1,\frac{1}{a}}\right)
\end{align*}
\end{theorem}
\begin{proof}
    Setting $y=u-v$ and $r(x)=\frac{1}{q}(\vert u_{x}\vert^q-\vert v_{x}\vert^q)$, from \eqref{23} we obtain
   \begin{equation}
       \left\{
\begin{aligned}
& y_{t}(x,t) +a(x) y_{xx}(x,t)=-d(x,t)y_{x}-c(x,t)y -b(x,t)r(x)\quad\text{ in } Q,\\
& y(x,t)=0 \hspace{8.5cm}\text{ on } \Sigma,
\label{hjeq1to4}
\end{aligned}
\right.
\end{equation}

\noindent{\bf Case 1:} Using the following key inequality
\begin{equation}\label{kineq}
\vert a^q-b^q \vert \leqslant q \left(a^{q-1}+b^{q-1}\right) \vert a-b\vert,
\end{equation}
for all $q \geqslant 1, \;\;a,b\geqslant 0$, we obtain
\begin{align*}
    \vert r(x)\vert =\frac{1}{q}\left\| u_{x}\vert^q-\vert v_{x}\vert^q\right\vert  &\leqslant \left( \vert u_{x}\vert^{q-1}+\vert v_{x}\vert^{q-1}\right)(\vert u_{x}\vert-\vert v_{x}\vert) \\&\leqslant \left( \vert u_{x}\vert^{q-1}+\vert v_{x}\vert^{q-1}\right)\vert y_{x}\vert\\& \leqslant C\vert y_{x}\vert.
\end{align*}
Then we apply lemma \eqref{1} to obtain
\begin{align*}
    &\int_{Q} \left(\frac{1}{a(x)}\vert y_{t}\vert^2+ a(x)\vert y_{xx}\vert^2+ s\lambda \varphi  \vert y_{x}\vert^2+ \frac{s^2\lambda^2\varphi^2}{a(x)}\vert u\vert^2\right)e^{2s\varphi}dxdt\\&\leqslant C  \int_{Q}s\varphi\frac{1}{a(x)}\vert b(x,t)r(x)+ d(x,t)y_{x} + c(x,t) y\vert^2e^{2s\varphi} dxdt \\&+C( s,\lambda)\left( s\lambda\varphi(T)\| y(x,T)\|_{\frac{1}{a}}^2+\| y(x,T)\|_{1,\frac{1}{a}}^2 \right)e^{2s\varphi(T)} \\&\qquad+C(s,\lambda)\left( s\lambda\| y(x,0)\|_{\frac{1}{a}}^2+\| y(x,0)\|_{1,\frac{1}{a}}^2\right)e^{2s}\\ &\leqslant C  \int_{Q}s\varphi \frac{b^2(x,t)}{a(x)}\vert y_{x}\vert^2 dxdt + \int_{Q}s\varphi{\frac{d^2(x,t)}{a(x)}}\vert y_{x}\vert^2+ C \int_{Q}s\varphi\frac{1}{a(x)}  \vert y\vert^2e^{2s\varphi} dxdt \\&+C( s,\lambda)\left( s\lambda\varphi(T)\| y(\cdot,T)\|_{\frac{1}{a}}^2+\| y(\cdot,T)\|_{1,\frac{1}{a}}^2 \right)e^{2s\varphi(T)}\\& \qquad+C(s,\lambda)\left( s\lambda\| y(\cdot,0)\|_{\frac{1}{a}}^2+\| y(\cdot,0)\|_{1,\frac{1}{a}}^2\right)e^{2s}.
\end{align*}
Using the fact that $\vert r(x) \vert \leqslant C \vert y_{x}\vert, \;\vert b(x,t)+ d(x,t)\vert\leqslant C \sqrt{a(x)} \;\; \forall (x,t)\in Q $ and $c\in L^{\infty}(Q)$. The first term of the right-hand side can be absorbed, and by choosing $s>0$ and $\lambda >0$ very large, we obtain
\begin{equation*}
    \begin{aligned}
    &\int_{Q} \left(\frac{1}{a(x)}\vert y_{t}\vert^2+ a(x)\vert y_{xx}\vert^2+ s\lambda \varphi \vert y_{x}\vert^2+ \frac{s^2\lambda^2\varphi^2}{a(x)}\vert u\vert^2\right)e^{2s\varphi}dxdt\\&
    \leqslant C( s,\lambda)\left( s\lambda\varphi(T)\| y(\cdot,T)\|_{\frac{1}{a}}^2+\| y(\cdot,T)\|_{1,\frac{1}{a}}^2 \right)e^{2s\varphi(T)}\\& \quad +C(s,\lambda)\left( s\lambda\| y(\cdot,0)\|_{\frac{1}{a}}^2+\| y(\cdot,0)\|_{1,\frac{1}{a}}^2\right)e^{2s}.
\end{aligned}
\end{equation*}
Then the same argument as in Theorem \ref{15} can complete the proof.
\smallskip

\noindent{\bf Case 2:} Since we have 
\begin{align*}
    \vert r(x) \vert \leqslant \left(\vert u_{x}\vert^{q-1}+\vert v_{x}\vert^{q-1}\right)\vert y_{x}\vert,
\end{align*}
then we apply Lemma \ref{lem1} to obtain
\begin{align*}
    &\int_{Q} \left(\frac{1}{a(x)}\vert y_{t}\vert^2+ a(x)\vert y_{xx}\vert^2+ s\lambda \varphi  \vert y_{x}\vert^2+ \frac{s^2\lambda^2\varphi^2}{a(x)}\vert u\vert^2\right)e^{2s\varphi}dxdt\\&\leqslant C  \int_{Q}s\varphi\frac{1}{a(x)}\vert b(x,t)r(x)+ d(x,t)y_{x} + c(x,t) y\vert^2e^{2s\varphi} dxdt \\&\qquad+C( s,\lambda)\left( s\lambda\varphi(T)\| y(\cdot,T)\|_{\frac{1}{a}}^2+\| y(\cdot,T)\|_{1,\frac{1}{a}}^2 \right)e^{2s\varphi(T)} \\&\qquad
    +C(s,\lambda)\left( s\lambda\| y(\cdot,0)\|_{\frac{1}{a}}^2+\| y(\cdot,0)\|_{1,\frac{1}{a}}^2\right)e^{2s}\\ &\leqslant C  \int_{Q}s\varphi \lvert \left( \frac{\vert u_{x}\vert^{q-1}}{\sqrt{a(x)}}+\frac{\vert u_{x}\vert^{q-1}}{\sqrt{a(x)}}\right)y_{x}\rvert^2 dxdt + \int_{Q}s\varphi{\frac{d^2(x,t)}{a(x)}}\vert y_{x}\vert^2\\& \qquad+ C \int_{Q}s\varphi\frac{1}{a(x)}  \vert y\vert^2e^{2s\varphi} dxdt +C( s,\lambda)\left( s\lambda\varphi(T)\| y(\cdot,T)\|_{\frac{1}{a}}^2+\| y(\cdot,T)\|_{1,\frac{1}{a}}^2 \right)e^{2s\varphi(T)}\\&\qquad+C(s,\lambda)\left( s\lambda\| y(\cdot,0)\|_{\frac{1}{a}}^2+\| y(\cdot,0)\|_{1,\frac{1}{a}}^2\right)e^{2s}.
\end{align*}

\end{proof}
\section{Identification for linear degenerate parabolic problem}\label{sec4}
We consider the following degenerate parabolic equation 
\begin{equation}
    \left\{
\begin{aligned}
& u_{t}(x,t) -a(x) u_{xx}(x,t)=0, &&\quad \text { in } Q,\\
& u(x,t)=0, &&\quad \text{ on } \Sigma,\\
&u(x,0)=f(x), &&\quad \text { in } \Omega.
\label{eq1to10}
\end{aligned}
\right.
\end{equation}
Recall that the backward problem consists of the determination of $f=u(\cdot,0)$ from the final state data $u_{T}(x)=u(x,T).$ For a given $u(\cdot,0)=f \in L^2_{\frac{1}{a}}(0,1)$, the parabolic problem \eqref{eq1to10} will be referred to as the direct problem.

The input-output operator $\Psi$ is defined as follows
$$\Psi(f)(x)=u(x,T;f), \qquad x\in (0,1).$$
Therefore, the backward problem can be written as an operator equation as follows
\begin{align*}
    \Psi(f)=u_{T}, \qquad f\in L^2_{\frac{1}{a}}(0,1) 
\end{align*}
\begin{lemma}\label{lemma1}
    The input-output operator
    $$\Psi\;:\; L^2_{\frac{1}{a}}(0,1) \ni f \longmapsto u(\cdot,T;f)\in L^2_{\frac{1}{a}}(0,1).$$
    is not surjective.
\end{lemma}
\begin{proof}
    To derive a contradiction, let's assume that $\Psi$ is surjective. Let $g \in L^2_{\frac{1}{a}}(0,1)\backslash H^1_{\frac{1}{a}}(0,1).$ Then, there exists an element $f \in L^2_{\frac{1}{a}}(0,1)$ such that $\Psi(f)=g,$ which implies $u(\cdot,T)=g,$ where $u$ is a solution of \eqref{eq1to10}. However, this contradicts the fact that $u(\cdot,T)\in  H^1_{\frac{1}{a}}(0,1) $ (see Theorem \ref{thmgenni}).
\end{proof}
From Lemma \ref{lemma1}, we conclude that the backward problem is ill-posed. Moreover, we define a quasi-solution of this inverse problem as a minimizer of the functional $J$ defined by
\begin{align*}
    J(f)=\frac{1}{2}\int_{0}^{1}\frac{1}{a(x)}\vert u(x,T;f)-u^{\delta}_{T}(x)\vert^2dx=\frac{1}{2}\|u(\cdot,T;f)-u^{\delta}_T\|^2_{\frac{1}{a}}, \;\; f\in W,
\end{align*}
where $u^{\delta}_T$ is a noisy data of $u_T$ such that $\|u_T-u^{\delta}_T\|_{\frac{1}{a}} \leqslant \delta$ for $\delta\geqslant 0$, and   $$W=\left\lbrace f \in L_{\frac{1}{a}}^2(0,1)\,:\; \; \|f\|_{\frac{1}{a}}\leqslant M\right\rbrace$$
is the set of admissible initial data of size $M$ (fixed positive constant). Obviously, $W$ is a bounded, closed and convex subset of $L_{\frac{1}{a}}^2(0,1).$

Due to the ill-posedness of the backward problem, one usually uses a regularization approach by introducing the Tikhonov functional
$$ J_{\epsilon}(f)= \frac{1}{2}\|u(\cdot,T;f)-u^{\delta}_T\|^2_{\frac{1}{a}}+\frac{\epsilon}{2}\|f\|^2_{\frac{1}{a}}, \qquad f\in W,$$
where $\epsilon>0$ is the regularizing parameter.

In the following lemma, we establish a conditional stability estimate using the logarithmic convexity method; see, e.g., \cite{Payne}, and \cite{Car18} for a survey.
\begin{lemma}
The solution of \eqref{eq1to10} satisfies 
    \begin{align}\label{convx}
        \|u(\cdot,t)\|_{\frac{1}{a}}\leqslant M^{1-\frac{t}{T}}\|u(\cdot,T)\|_{\frac{1}{a}}^{\frac{t}{T}} \qquad \forall t \in [0,T],
    \end{align}
for all $f\in W.$
\end{lemma}
\begin{proof} Let $f \neq 0$, otherwise the estimate is trivial.
    We have 
   \begin{equation*}
       \begin{aligned}
       \frac{d}{dt}\|u(\cdot,t)\|_{\frac{1}{a}}^2&=2(u_{t}(\cdot,t),u(\cdot,t))_{\frac{1}{a}}\\&=2(au_{xx}(\cdot,t),u(\cdot,t))_{\frac{1}{a}}\\& =2\int_0^1u_{xx}udx \\&=-2\int_0^1 u_x^2 dx
   \end{aligned}
   \end{equation*}
   and 
 \begin{equation*}
     \begin{aligned}
         \frac{d^2}{dt^2}\|u(\cdot,t)\|_{\frac{1}{a}}^2&=-4\int_0^1 (u_t)_x u_x dx\\&=4 \int_0^1 u_t u_{xx} dx \\&=4 \int_0^1 a(u_{xx})^2 dx\\&= 4\|au_{xx}\|_{\frac{1}{a}}^2.
     \end{aligned}
 \end{equation*}
By the Cauchy–Schwarz inequality, we obtain
 \begin{align*}
    \hspace{-0.15cm} \left( \frac{d^2}{dt^2}\|u(\cdot,t)\|_{\frac{1}{a}}^2\right)\|u(\cdot,t)\|_{\frac{1}{a}}^2-&\left( \frac{d}{dt}\|u(\cdot,t)\|_{\frac{1}{a}}^2\right)^2\\&=4\left(\|au_{xx}\|_{\frac{1}{a}}^2\|u(\cdot,t)\|_{\frac{1}{a}}^2-(au_{xx}(\cdot,t),u(\cdot,t))_{\frac{1}{a}}^2\right)\geqslant 0.
 \end{align*}
 Since,
 $$\frac{d^2}{dt^2}\log\left(\|u(\cdot,t)\|_{\frac{1}{a}}^2\right)=\frac{\left( \frac{d^2}{dt^2}\|u(\cdot,t)\|_{\frac{1}{a}}^2\right)\|u(\cdot,t)\|_{\frac{1}{a}}^2-\left( \frac{d}{dt}\|u(\cdot,t)\|_{\frac{1}{a}}^2\right)^2}{\|u(\cdot,t)\|_{\frac{1}{a}}^4}\geqslant 0,$$
$\log\left(\|u(\cdot,t)\|_{\frac{1}{a}}^2\right)$ is a convex function of $t$, which gives
 $$\log\left(\|u(\cdot,t)\|_{\frac{1}{a}}^2\right)\leqslant \left(1-\frac{t}{T}\right)\log\left(\|u(\cdot,0)\|_{\frac{1}{a}}^2\right)+\frac{t}{T}\log\left(\|u(\cdot,T)\|_{\frac{1}{a}}^2\right).$$
Thus, we deduce \eqref{convx}.
\end{proof}
\subsection{Fréchet differentiability of the Tikhonov functional}
We recall the following definition:
\begin{definition}
    Let $f \in L^{2}_{\frac{1}{a}}(0,1)$. A function $u$ is a weak solution of \eqref{eq1to10} if 
    $$ u \in C\left([0,T];L^{2}_{\frac{1}{a}}(0,1)\right)\cap L^2\left(0,T; H^1_{\frac{1}{a}}(0,1)\right)$$
    and satisfies
\begin{equation*}
    \begin{aligned}
        &\int_{0}^1\frac{1}{a(x)}u(x,T)v(x,T) dx-\int_{0}^1\frac{1}{a(x)}u(x,0)v(x,0) dx-\int_{Q}\frac{1}{a(x)}u(x,t)v_t(x,t)dxdt\\&= -\int_{Q}u_x(x,t)v_x(x,t)dxdt
    \end{aligned}
\end{equation*}
for all $v \in H^1\left(0,T;L^{2}_{\frac{1}{a}}(0,1)\right)\cap L^2\left(0,T; H^1_{\frac{1}{a}}(0,1)\right).$
\end{definition}
\begin{lemma}
    Denote by $u(x,t;f)$ the weak solution of the parabolic degenerate problem \eqref{eq1to10} corresponding to the unknown initial datum $ f\in W.$ Then functional $J$ is Fr\'echet differentiable at $f$ and  the Fr\'echet gradient is given by
    \begin{align}\label{34}
         J'(f)(x)= \psi(x,0;f)\qquad \text{for a.e. }\quad x\in (0,1),
    \end{align}
    where $\psi(x,t;f)$ is the weak solution of the adjoint problem
    \begin{equation}
        \left\{
\begin{aligned}
& \psi_{t}(x,t) = -a(x) \psi_{xx}(x,t), &&\quad \text { in } Q,\\
& \psi=0, &&\quad \text { on } \Sigma,\\
& \psi(x,T)=u(x,T;f)-u^{\delta}_{T}(x), &&\quad \text { in } \Omega.
\label{eq1to5}
\end{aligned}
\right.
\end{equation}
\end{lemma}
\begin{proof}
    Let $f,\, f+\delta f \in W$ be given initial data and $u(x,t;f), \, u(x,t;f+\delta f)$ be corresponding solutions of (\ref{eq1to10}). Then 
    $\delta u(x,t;f)=u(x,t;f+\delta f)-u(x,t,f)$
    is the weak  solution of the following system
\begin{equation}
    \left\{
\begin{aligned}
&\delta u_{t}(x,t) = a(x) \delta u_{xx}(x,t),&&\quad \text { in } Q\\
&\delta u(x,t)=0, &&\quad \text { on } \Sigma,\\
& \delta u(x,0)=\delta f(x), &&\quad \text { in } \Omega.
\label{eq1to6}
\end{aligned}
\right.
\end{equation}
Let us calculate the increment $\delta J(f)=J(f+\delta f)-J(f)$ of functional (\ref{2})
\begin{align*}
    \delta J(f)&=\frac{1}{2}\|u(\cdot,T;f+\delta f)-u^{\delta}_T\|_{\frac{1}{a}}-\frac{1}{2}\|u(\cdot,T;f)-u_T^{\delta}\|_{\frac{1}{a}}\\&=\frac{1}{2}\int_{0}^1\frac{1}{a(x)}\left[(u(x,T;f+\delta f)-u^{\delta}_T(x))^2-(u(x,T;f)-u^{\delta}_T(x))^2\right].
\end{align*}
Using the identity
$$\frac{1}{2}\left[(r-t)^2-(s-t)^2\right]=(s-t)(r-s)+\frac{1}{2}(r-s)^2, \;r,s,t\in \mathbb{R},$$
we obtain
\begin{align*}
    \delta J(f)=\int_{ 0}^1\frac{1}{a(x)}\left[u(x,T;f)-u^{\delta}_{T}(x)\right]\delta u(x,T;f)dx + \frac{1}{2}\int_{0}^1 \frac{1}{a(x)} \left[\delta u(x,T;f)\right]^2 dx.
\end{align*}
We have 
\begin{equation*}
\begin{aligned}
&\int_{ 0}^1\frac{1}{a(x)}\left[u(x,T;f)-u^{\delta}_{T}(x)\right]\delta u(x,T;f)dx \\ & = \int_{ 0}^1\frac{1}{a(x)}\psi(x,T)\delta u(x,T) dx\\ &= \int_{ 0}^1 \left\lbrace \int_{0}^T \frac{1}{a(x)}\left( \psi(x,t)\delta u(x,t)\right)_{t} dt\right\rbrace + \int_{ 0}^1\frac{1}{a(x)}\psi(x,0)\delta u(x,0) dx\\ &= \int_{ 0}^1 \int_{0}^T\frac{1}{a(x)}\left( \psi_{t}\delta u + \psi \delta u_{t}\right) dt dx + \int_{ 0}^1\frac{1}{a(x)}\psi(x,0)\delta u(x,0) dx \\ & = \int_{ 0}^1 \int_{0}^T\frac{1}{a(x)}\left( -a(x)\psi_{xx}\delta u + \psi a(x)\delta u_{xx}\right) dt dx + \int_{ 0}^1\frac{1}{a(x)}\psi(x,0)\delta u(x,0) dx \\&= \int_{0}^T\left( -\psi_{xx}\delta u + \psi\delta u_{xx}\right) dt dx + \int_{ 0}^1\frac{1}{a(x)}\psi(x,0)\delta u(x,0) dx \\&= \int_{ 0}^1\frac{1}{a(x)}\psi(x,0)\delta u(x,0) dx.
\end{aligned}
\end{equation*}
This implies 
\begin{equation*}
    \begin{aligned}
        \delta J(f)=&\int_{ 0}^1\frac{1}{a(x)}\psi(x,0)\delta u(x,0) dx+\frac{1}{2}\int_{0}^1 \left[\delta u(x,T;f)\right]^2 dx\\&
        =(J'(f),\delta f)_{\frac{1}{a}}+\frac{1}{2}\int_{0}^1 \frac{1}{a(x)} \left[\delta u(x,T;f)\right]^2 dx.
    \end{aligned}
\end{equation*}
Using the result of Theorem \ref{thmgenni} yields 
$$ \int_{0}^1 \frac{1}{a(x)}\left[\delta u(x,T;f)\right]^2 dx\leqslant \vert \vert \delta f\vert \vert_{\frac{1}{a}}^2. $$
Hence, 
$$\delta J(f)=(J'(f),\delta f)+o\left(\vert \vert \delta f\vert \vert_{\frac{1}{a}}^2\right).$$
\end{proof}

In optimization, it is classical that the Lipschitz continuity of the Fréchet gradient $J'$ yields fast numerical results. This is proven in the following lemma.
\begin{lemma}
The Fréchet derivative $J'$ is Lipschitz continuous with Lipschitz constant $L=1$, i.e.,
    $$ \vert \vert J'(f + \delta f )-J'(f)\vert\vert_{\frac{1}{a}}\leqslant  \vert \vert \delta f\vert \vert_{\frac{1}{a}} \qquad \forall f,\, f+\delta f \in W.$$
\end{lemma}
\begin{proof}
    By definition 
    $$\vert\vert J'(f+\delta f)-J'(f)\vert\vert_{\frac{1}{a}}=\vert\vert\delta \psi(\cdot,0;f)\vert\vert_{\frac{1}{a}},$$
    where
    $\delta\psi(x,t;f)= \psi(x,t;f+\delta f)-\psi(x,t;f)$ solves
\begin{equation}
    \left\{
\begin{aligned}
& \delta\psi_{t}(x,t) = -a(x) \delta\psi_{xx}(x,t),&&\quad \text { in } Q,\\
& \delta\psi(x,t)=0, &&\quad \text { on } \Sigma,\\
& \delta\psi(x,T)=\delta u(x,T;f), &&\quad \text { in } \Omega,
\label{eq1to7}
\end{aligned}
\right.
\end{equation}
and $\delta u(x,t;f)$ is the weak solution of the auxiliary problem
\begin{equation}
    \left\{
\begin{aligned}
& \delta u_{t}(x,t) = a(x) \delta u_{xx}(x,t),&&\quad \text { in } Q,\\
& \delta u(x,t)=0, &&\quad \text { on } \Sigma,\\
& \delta u(x,0)=\delta f(x), &&\quad \text { in } \Omega.
\label{eq1to8}
\end{aligned}
\right.
\end{equation}
We use the energy identity
\begin{equation*}
    \begin{aligned}
        \frac{1}{2} \frac{d}{dt}\int_{0}^1 \frac{1}{a(x)}\left(\delta \psi(x,t)\right)^2 dx&= \int_{0}^1 \frac{1}{a(x)}\delta \psi \delta \psi_{t} dx\\& =\int_{0}^1 \frac{1}{a(x)}\delta \psi a(x) \delta \psi_{xx} dx\\&= -\int_{0}^1 \delta \psi \delta \psi_{xx}\\&= \int_{0}^1 (\delta \psi_{x})^2 dx \geqslant 0.
    \end{aligned}
\end{equation*}
Integrating on $[0,T]$ and using $\delta \psi(x,T;f)=\delta u(x,T;f)$, we obtain
$$\int_{0}^1 \frac{1}{a(x)}\left(\delta \psi(x,0;f)\right)^2 dx \leqslant \int_{0}^1 \frac{1}{a(x)}\left(\delta \psi(x,T;f)\right)^2 dx=  \int_{0}^1 \frac{1}{a(x)}\left(\delta u(x,T;f)\right)^2 dx.$$
Next, we employ the following equality
\begin{equation*}
  \begin{aligned}
        \frac{1}{2} \frac{d}{dt}\int_{0}^1 \frac{1}{a(x)}\left(\delta u(x,t)\right)^2 dx&= \int_{0}^1 \frac{1}{a(x)}\delta u \delta u_{t} dx\\& =\int_{0}^1 \frac{1}{a(x)}\delta u a(x) \delta u_{xx} dx\\&= \int_{0}^1 \delta u\delta u_{xx}\\&=- \int_{0}^1 (\delta u_{x})^2 dx \leqslant 0.
  \end{aligned}  
\end{equation*}
Integrating on $[0,T]$, we obtain 
$$\int_{0}^1 \frac{1}{a(x)}\left(\delta u(x,T;f)\right)^2 dx \leqslant \int_{0}^1 \frac{1}{a(x)}\left(\delta u(x,0;f)\right)^2 dx= \int_{0}^1 \frac{1}{a(x)}\vert\delta f\vert^2 dx.$$
The above estimates yield
$$\vert\vert\delta \psi(\cdot,0;f)\vert\vert_{\frac{1}{a}}\leqslant  \vert\vert\delta f\vert\vert_{\frac{1}{a}}.$$
\end{proof}

We introduce the Landweber iteration given by
\begin{align}\label{44}
     f_{n+1}=f_{n}-\alpha_n J'(f_n), \; n=0,1,2, \dots
\end{align}
where $f_0$ is a given initial iteration in $W$ and $ \alpha_n$ is the iteration parameter defined as follows
$$ g_n(\alpha_n):=\inf_{\alpha \geq 0}g_n(\alpha), \; g_n(\alpha):=J(f_n-\alpha J'(f_n)),\, n=0,1,2 \dots$$
\begin{remark} Some remarks are in order:
    \begin{itemize}
        \item[(i)] We can use a similar proof of Lemma 4.3 and Corollary 4.1 in \cite{Hassan2007} to show the monotonicity and convergence of the sequence $ \left\lbrace J(f_n)\right\rbrace$ where $f_n$ is defined by equation \eqref{44}. 
        \item[(ii)] As mentioned in \cite{Hassan2007}, the choice of the parameter $\alpha_n$ can be challenging in many situations. However, when we have the Lipschitz continuity of the gradient $J'(f)$, we can estimate this parameter by the Lipschitz constant $L=1$ as follows:
        $$ 0 <\lambda_0 \leqslant \alpha_n \leqslant \frac{2}{1+2\lambda_1},$$
        where $\lambda_0, \lambda_1>0$ are arbitrary parameters.
\item[(iii)] If the step size parameter $\alpha_n=\alpha >0$ for all $n$, the optimal value of $\alpha$ is $\alpha_{*}=1.$
    \end{itemize}
\end{remark}
We can adopt the argument in Lemma 4.3 and  Corollary 4.1 of \cite{Hassan2007} to prove:
\begin{lemma}
    Let $(f_n)$ be the sequence defined by iteration \eqref{44}. If the step size parameter $\alpha_n=\alpha_* >0$ for all $n\in \mathbb N$, where $\alpha$ is a constant, we obtain
    \begin{itemize}
        \item[(i)]
        The sequence $(J(f_n))$ is monotone decreasing and convergent. Moreover,
        $$\lim\limits_{n \rightarrow \infty}\|J'(f_n)\|_{\frac{1}{a}}=0,$$
    and
    \begin{align*}
        \|f_{n+1}-f_n\|_{\frac{1}{a}}^2\leqslant 2\left[J(f_n)-J(f_{n+1})\right].
    \end{align*}
    \item[(ii)] For any initial data $f_0 \in W$, the sequence $f_n $ converges weakly in $L^2_{\frac{1}{a}}(0,1)$ to a quasi-solution $f_* \in W_*$. Moreover, for the rate of convergence of the sequence $ (J(f_n))$  can be estimated as follows :
    $$ 0 \leqslant J(f_n)- J(f_*) \leqslant \frac{2\lambda^2}{n} \qquad \forall n \in \mathbb{N},$$
    where $\lambda:=\sup \{\|f_1-f_2\|_{\frac{1}{a}}, \; f_1, f_2\in W\}.$
    \end{itemize}
\end{lemma}
\subsection{Existence and uniqueness}
In this section, we will study the existence and uniqueness of the quasi-solution using the calculus of variations.
\begin{lemma}\label{lemmaexi}
    For $J\in C^1(W)$, the following equality holds
    $$ \langle J'(f+\delta f)-J'(f), \delta f \rangle_{\frac{1}{a}}=\|\delta u(\cdot,T,f)\|^2_{\frac{1}{a}} \qquad \forall f,\, f+\delta f \in W.$$
\end{lemma}
\begin{proof}
    Let $\delta \psi$ be the solution of \eqref{eq1to7}. Using gradient formula \eqref{34}, we obtain
    \begin{equation*}
        \begin{aligned}
    &\langle J'(f+\delta f)-J'(f), \delta f\rangle_{\frac{1}{a}}=\int_{0}^1 \frac{1}{a(x)}\delta \psi(\cdot,0)\delta u(\cdot,0) dx\\&=\int_{ 0}^1 \left\lbrace \int_{0}^T \frac{1}{a(x)}\left( \delta\psi(x,t)\delta u(x,t)\right)_{t} dt\right\rbrace dx + \int_{ 0}^1\frac{1}{a(x)}\delta\psi(x,T)\delta u(x,T) dx\\ &= \int_{ 0}^1 \int_{0}^T\frac{1}{a(x)}\left( \delta\psi_{t}\delta u + \delta\psi \delta u_{t}\right) dt dx + \int_{ 0}^1\frac{1}{a(x)}\delta\psi(x,T)\delta u(x,T) dx \\ & = \int_{ 0}^1 \int_{0}^T\frac{1}{a(x)}\left( -a(x)\delta\psi_{xx}\delta u + \delta\psi a(x)\delta u_{xx}\right) dt dx + \int_{ 0}^1\frac{1}{a(x)}\delta\psi(x,T)\delta u(x,T) dx \\&= \int_{0}^T\left( -\delta\psi_{xx}\delta u + \delta\psi\delta u_{xx}\right) dt dx + \int_{ 0}^1\frac{1}{a(x)}\delta u(x,T)\delta u(x,T) dx \\& =  \int_{ 0}^1\frac{1}{a(x)}\delta u(x,T)^2 dx .
        \end{aligned}
    \end{equation*}
\end{proof}
\begin{remark}
    The monotonicity of the derivative $J'$ in Lemma \ref{lemmaexi} implies the convexity of the functional $J$. Moreover, using Theorem 25.~C in \cite{ZEI2004}, we obtain the following result.
\end{remark}
\begin{lemma}
   The functional $J$ is continuous and convex on the closed convex set $W$. Then $J$ has at least a minimum $f_*$, i.e.,
    $$ J(f_*)= \min_{f \in W} J(f).$$
\end{lemma}
As a direct consequence of Corollary 25.15 in \cite{ZEI2004}, we have the following sufficient condition for uniqueness.
\begin{lemma}
    If we have 
    $$\|\delta u(\cdot,T,f)\|^2_{\frac{1}{a}}>0\qquad \forall f \in W_{0},$$
    where $W_0 \subset W$ is a closed convex subset. Then $J$ has at most one minimum.
\end{lemma}

\section{Numerical identification for Degenerate viscous Hamilton-Jacobi equation}\label{sec5}
In this section, we develop some numerical algorithms to reconstruct the unknown initial data from the measured final state.
First, we use the Conjugate Gradient (CG) method for the linear equation, while for the nonlinear viscous Hamilton-Jacobi equation, we employ Van Cittert's iteration, which turned out to be more suitable for handling the nonlinearity.

\subsection{Linear degenerate parabolic equation}\label{s1}
We consider the following degenerate parabolic equation 
\begin{equation}
    \left\{
\begin{aligned}\label{1}
& u_{t}(x,t) -a(x) u_{xx}(x,t)=0, &&\quad \text { in } Q\\
& u(x,t)=0, &&\quad \text{ on } \Sigma,\\
&u(x,0)=f(x), &&\quad \text { in } \Omega.
\end{aligned}
\right.
\end{equation}

\subsubsection{\textbf{Discretization of the problem}}
We approximate the solution of the equation \eqref{1} using the finite difference method. We discretize the space-time domain $[0,1]\times [0,T]$ into a uniform grid $x_{j}=jdx $ for $j=0,\cdots, N_x+1$, with a spatial step size $dx=\frac{1}{N_x+1},$ and $t_n=ndt$ for $n=0,\cdots , N_t+1,$ with a time step size $dt=\frac{T}{N_t+1}$. We denote by $ u^n_{j}$ the approximate solution at node $x_{j}$ at time $t_n,$ and  $ a_{j}=a(x_{j}) $. The first-order and second-order derivatives of $u $ in \eqref{1} are approximated  by:
$$ \partial_t u (x_j,t_n) \approx  \frac{u_{j}^{n+1}-u_{j}^{n}}{dt} $$
and 
$$\partial_{xx}u (x_j,t_n) \approx \frac{1}{2}\left(\frac{u_{j+1}^{n+1}-2u_j^{n+1}+u_{j-1}^{n+1}}{dx^2}+\frac{u_{j+1}^{n}-2u_j^{n}+u_{j-1}^{n}}{dx^2}\right). $$
Then we consider the Crank–Nicolson discretization to obtain:
$$ \frac{u_{j}^{n+1}-u_{j}^{n}}{dt} = \frac{a_{j}}{2}\left(\frac{u_{j+1}^{n+1}-2u_j^{n+1}+u_{j-1}^{n+1}}{dx^2}+\frac{u_{j+1}^{n}-2u_j^{n}+u_{j-1}^{n}}{dx^2}\right).$$
This numerical scheme can be written as follows
$$-\gamma_ju_{j-1}^{n+1}+(1+2\gamma_j)u_j^{n+1}-\gamma_ju_{j+1}^{n+1}=\gamma_ju_{j-1}^n+ (1-2\gamma_j)u_j^n+\gamma_ju^n_{j+1},$$
where $\gamma_j=\dfrac{a_j dt}{2dx^2}.$ Then, we obtain the following algebraic equation
$$ Au^{n+1}=Bu^{n}, \quad n=0,\cdots, N_t,$$
where $u^n=(u^n_1, u_2^n,\cdots, u_{N_x}^n)^t$, and
$$A=\begin{bmatrix}
1 + 2\gamma_1 & -\gamma_1 & 0 & \cdots & 0 \\
-\gamma_2 & 1 + 2\gamma_2& -\gamma_2 & \ddots & \vdots \\
0 & -\gamma_3 & 1 + 2\gamma_3 & \ddots & 0 \\
\vdots & \ddots & \ddots & \ddots & -\gamma_{N-1} \\
0 & \cdots & 0 & -\gamma_{N} & 1 + 2\gamma_N
\end{bmatrix},
$$     
$$B=\begin{bmatrix}
1 - 2\gamma_1 & \gamma_1 & 0 & \cdots & 0 \\
\gamma_2 & 1 - 2\gamma_2& \gamma_2 & \ddots & \vdots \\
0 & \gamma_3 & 1 - 2\gamma_3 & \ddots & 0 \\
\vdots & \ddots & \ddots & \ddots & \gamma_{N-1} \\
0 & \cdots & 0 & \gamma_{N} & 1 - 2\gamma_N
\end{bmatrix}.
$$    
\subsubsection{\textbf{Numerical experiments and results}}
In this section, we introduce the Conjugate Gradient (CG) method for the numerical reconstruction of initial data.
 
We implement the gradient formula of $J_{\epsilon}$ (see \eqref{34}) for the backward problem in the following CG algorithm.
\smallskip

\begin{algorithm}[H]
        \caption{Conjugate Gradient Algorithm}\label{alg}
        \begin{algorithmic}[1]
            \State Set $n=0$ and choose an initial guess $f_0$;
            \State Solve the direct problem \eqref{1} to obtain $u(t,x;f_0);$
            \State Knowing the computed $ u(T,x;f_0)$ and the measured $u^{\delta}_T,$ solve the adjoint problem  \eqref{eq1to5} to obtain $\psi(t,x;f_0);$
            \State Compute the initial descent direction $p_0=J_{\epsilon}'(f_0)$ using \eqref{34};
            \State Solve the direct problem \eqref{1} with initial datum $p_n$ to obtain the solution of $\Psi(p_n);$
            \State Compute the relaxation parameter $ \alpha_n =\frac{\|J_{\epsilon}'(f_n)\|^2_{\frac{1}{a}}}{\|\Psi p_n\|^2_{\frac{1}{a}}+\epsilon\| p_n\|^2_{\frac{1}{a}}};$
            \State Update the iteration $f_{n+1}=f_n -\alpha_n p_n ;$
            \State Stop the iteration process if the stopping criterion $J_{\epsilon}(f_{n+1})< \epsilon'+\rho \mathcal{E}$ holds,  where $\rho>1,$ $\mathcal{E}=\frac{1}{2} \|u_{\text{exact}}-u^\delta\|_{\frac{1}{a}}^2$ is the magnitude of added noise, and $\epsilon'$ is the tolerance imposed if there's no noise.\\
            Otherwise, set $n=n+1,$ compute
$$ \gamma_n=\frac{\|J_{\epsilon}'(f_n)\|^2_{\frac{1}{a}}}{\|J_{\epsilon}'(f_{n-1})\|^2_{\frac{1}{a}}}\;\; \text{and}\;\; p_n=J_{\epsilon}'(f_n)+\gamma_n p_{n-1},$$
and go to step $5$;
\end{algorithmic}
\end{algorithm}
\smallskip
As we know, measurements always contain some level of random noise. In all the following numerical experiments, we perturb the exact data as follows
\begin{equation}\label{ndata}
    u^{\delta}=u_{\text{exact}}+p\times\beta\times \max_{[0,1]}\vert u_{\text{exact}}\vert \times \texttt{random.rand()},
\end{equation}
instead of $u_{\text{exact}}$ to generate noisy data, where $p$ stands for the percentage of the noise level, $\beta \in L^2_{\frac{1}{a}}(0,1)$ and the function \texttt{random.rand()} produces random real numbers  uniformly distributed in the interval $[0,1)$. In all the numerical examples, we take $T=0.03$, $N_x=N_t=100$, $\epsilon=10^{-6}$ and $\rho=1.1.$

\begin{remark}
All the integrals involved in the algorithms are computed numerically using the well-known \textit{Gauss-Kronrod} quadrature. The direct and adjoint problems are solved using the finite difference method, and the stopping criteria are selected based on the discrepancy principle \cite{Alifanov2012,Ozisik2000}.
\end{remark}

\begin{example}\label{example1}
We take 
$$f(x)= \sqrt{x(1-x)}, \quad x\in (0,1), \; a(x)=\sin(\pi x), \quad \beta(x)=\sqrt{x(1-x)}.$$
We initialize the iteration with $f_0=0$ and fix the tolerance $\epsilon'=10^{-7}$. The algorithm terminates after $n\in \left\lbrace 101,6,5,2\right\rbrace$ iterations for different noise levels $p\in \left\lbrace0\%, 1\%, 3\%, 5\%\right\rbrace $ (see Figure \ref{fig11}), respectively. The above stopping iteration numbers are obtained based on the discrepancy principle. Moreover, Figure \ref{fig2} compares the exact function with the recovered solution at each iteration for a noise level of $p = 1\%$. The recovered function progressively approaches the exact one.

\begin{figure}[H]
\centering
\begin{subfigure}[t]{0.48\textwidth}
    \centering
    \includegraphics[width=\textwidth]{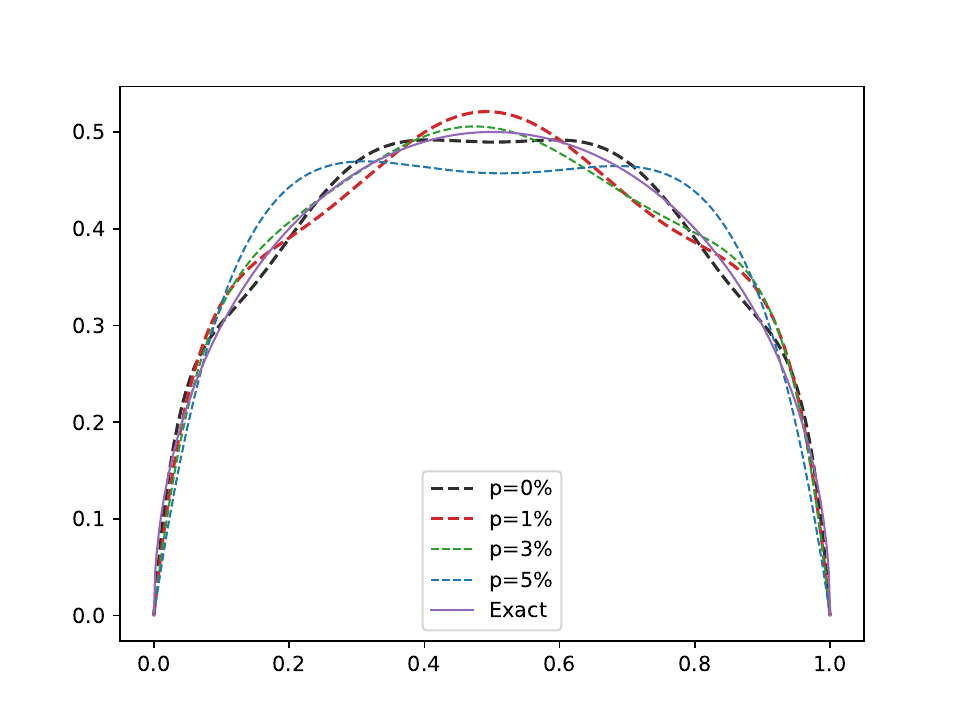}
     \caption{Exact and recovered initial data with\\ $p\in \left\lbrace 0\%,1 \%, 3\%, 5\%\right\rbrace $ and $\epsilon'=10^{-7}$}%
    \label{fig11}%
\end{subfigure}
\begin{subfigure}[t]{0.48\textwidth}
    \centering
    {{\includegraphics[width=\textwidth]{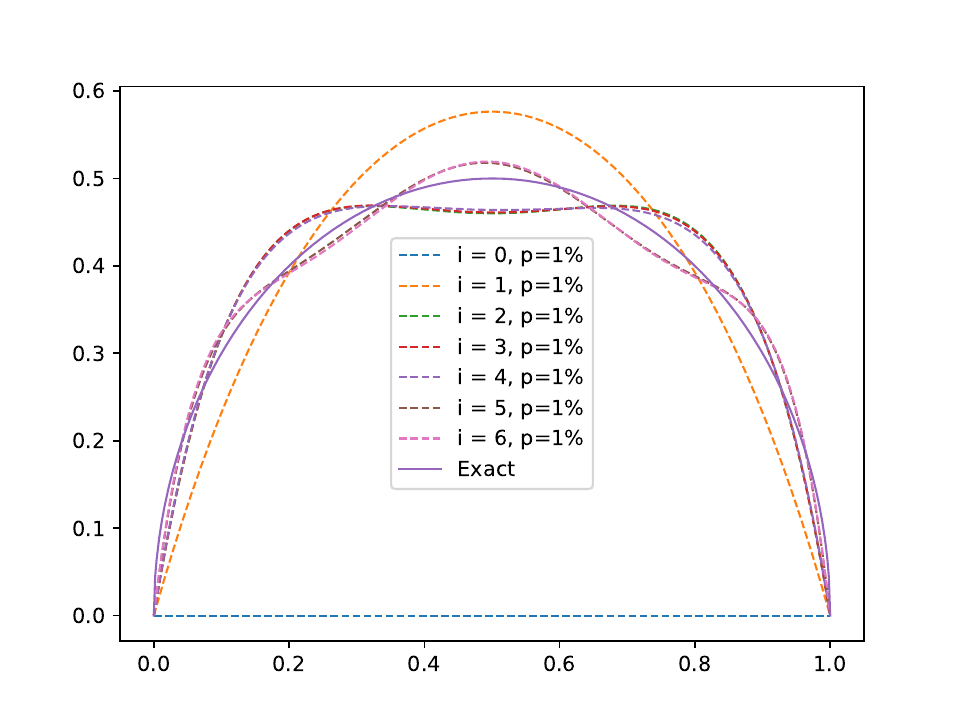}}}%
    \caption{Iterative recovery of the exact function with noise level $p=1\%$}%
    \label{fig2}%
\end{subfigure}
\caption{Recovered initial data: (a) for multiple noise levels, (b) iterative recovery for $p=1\%$}
\end{figure}
\vspace{-0.5em}
Table~\ref{tab:error_noise} presents the error values and the corresponding number of iterations for different noise levels. In the noise-free case (0\%), the algorithm achieves a low error of \(1.72 \times 10^{-3}\) after 101 iterations. When 1\% noise is introduced, the algorithm reaches an error of \(3.11 \times 10^{-3}\) and terminates after 6 iterations. As the noise level increases to $3\%$ and $5\%$, the error also increases to \(6.51 \times 10^{-3}\) and \(6.54 \times 10^{-3}\), respectively. While the number of iterations remains the same for $3\%$ noise, it drops to 2 for $5\%$, likely due to early stopping based on the discrepancy principle.

\begin{table}[h!]
\centering
\begin{tabular}{|c|c|c|}
\hline
\textbf{Noise Level (\%)} & \textbf{Error} & \textbf{Number of Iterations} \\
\hline
0 & \(1.72 \times 10^{-3} \) & 101 \\
1.0 & \(3.11 \times 10^{-3} \) & 6 \\
3.0 & \(6.51 \times 10^{-3} \) & 5 \\
5.0 & \(6.54 \times 10^{-3} \) & 2 \\
\hline
\end{tabular}
\caption{Error values and number of iterations at different noise levels.}
\label{tab:error_noise}
\end{table}

Figure \ref{fig:objective_decrease} illustrates the evolution of the Tikhonov functional $J_{\epsilon}(f_n)$ during the Conjugate Gradient iterations. We see that the values of the functional decrease rapidly, indicating convergence. The use of a logarithmic scale highlights the sharp drop after the first iteration and the stabilization in subsequent steps.
\begin{figure}[H]
    \centering
    \includegraphics[scale=0.4]{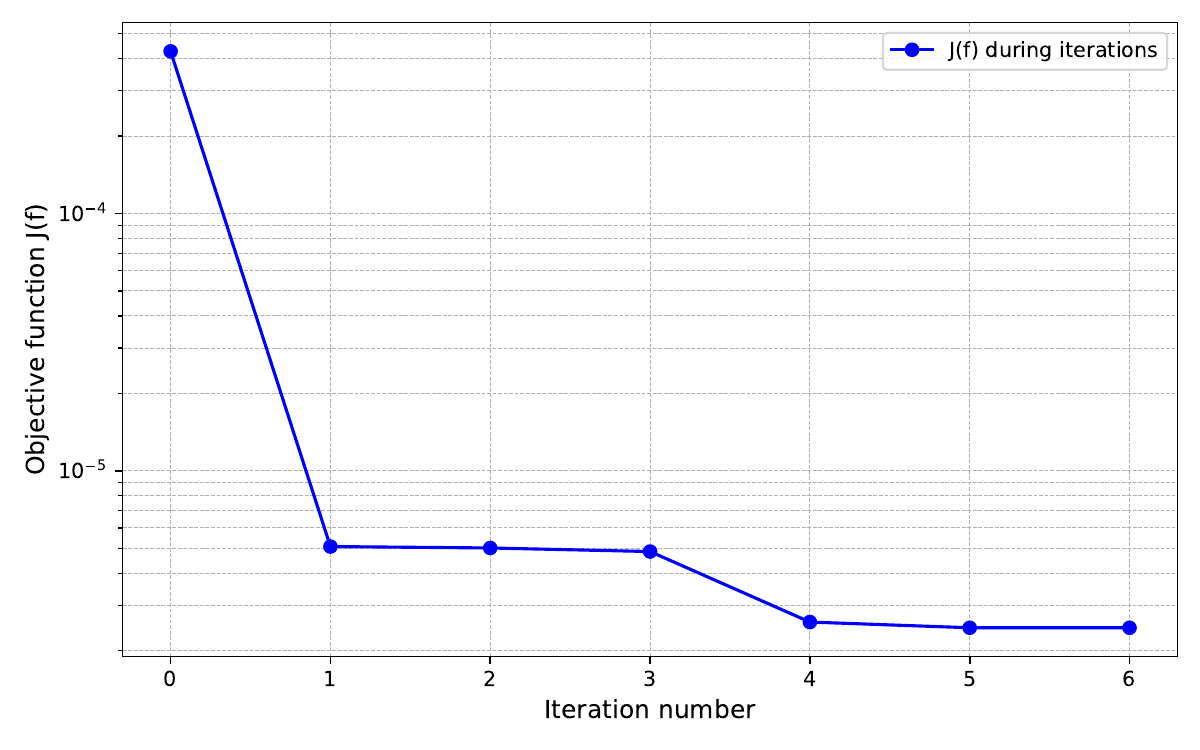}
    \caption{Decrease of the functional during the Conjugate Gradient process}
    \label{fig:objective_decrease}
\end{figure}
\end{example}

\begin{example}
We consider an exact initial datum that is piecewise discontinuous in $(0,1)$ and given by
$$f(x) =
\begin{cases}
1 & \quad\text{if } x\in [0.2,0.8], \\
0 & \quad\text{otherwise}, 
\end{cases}$$
and we take $$a(x)=x^4(1-x)^4, \qquad \beta(x)=(1-\cos(x))^2(e^{1-x}-1)^2.$$
We initialize the iteration with $f_0=0$ and fix the tolerance $\epsilon'=10^{-7}$. As in Example \ref{example1}, the algorithm terminates after $n=5$ iterations for all different noise levels $p\in \left\lbrace 0\%,1\%, 3\%, 5\%\right\rbrace $ (see Figure \ref{fig1}). Additionally, Figure \ref{fig111} displays the evolution of the recovered function with a noise level of $p = 1\%$. The solution becomes increasingly close to the exact function as the iterations progress.

\begin{figure}[h]
\centering
\begin{subfigure}[t]{0.48\textwidth}
    \centering
    {{\includegraphics[width=\textwidth]{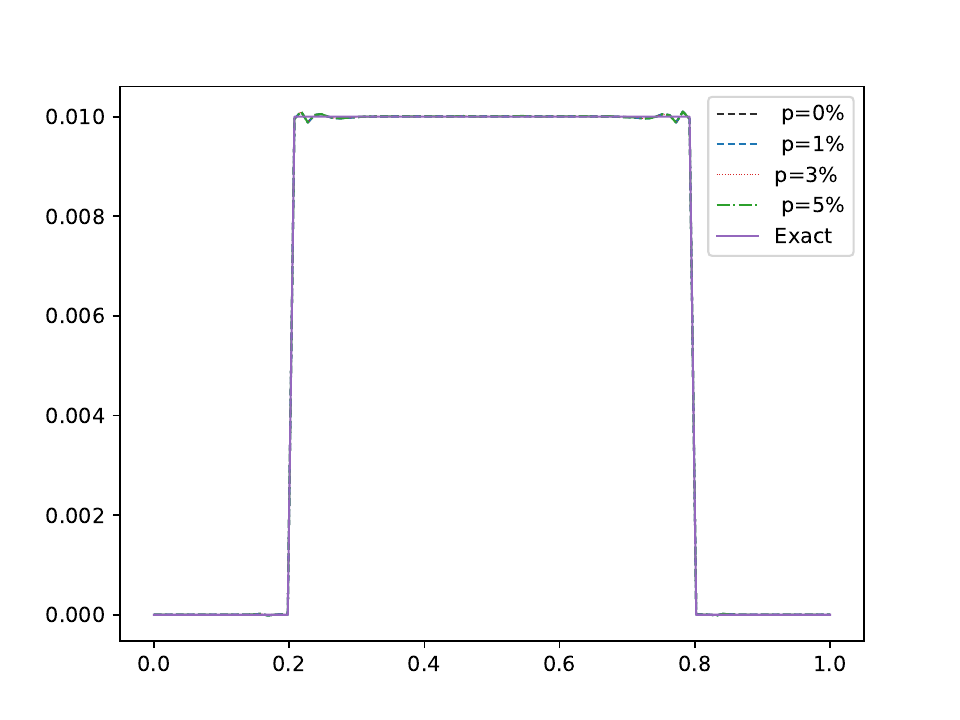} }}%
    \caption{Exact and recovered initial data with \\$p\in \left\lbrace 0\%, 1\%, 3\%, 5\%\right\rbrace $ and $\epsilon'=10^{-7}$}%
    \label{fig1}%
\end{subfigure}
\begin{subfigure}[t]{0.48\textwidth}
     \centering
    {{\includegraphics[width=\textwidth]{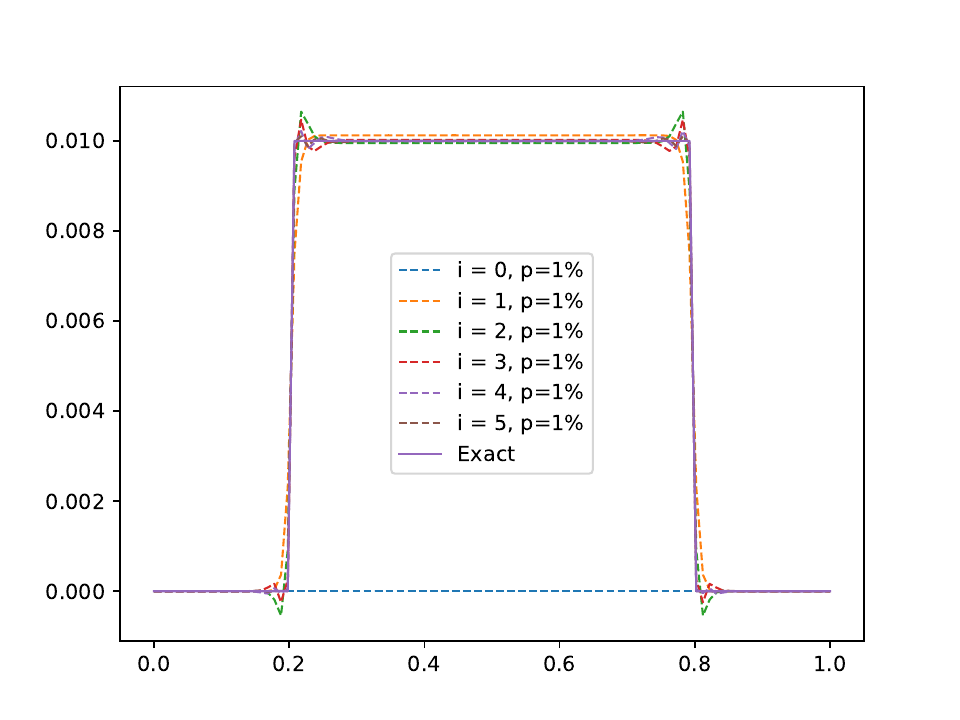} }}%
    \caption{Iterative recovery of the exact function with noise level $p=1\%$}%
    \label{fig111}%
\end{subfigure}
\caption{Recovered initial data: (a) for multiple noise levels, (b) iterative recovery for $p=1\%$}
\end{figure}
\vspace{-0.5em}
Table~\ref{tab:error_noise3} presents the error values and the corresponding number of iterations for different noise levels under discrepancy-based stopping. In all cases, from 0\% to 5\% noise, the number of iterations remains constant at 5, indicating that the stopping criterion is consistently met at the same point. The error values show only a slight increase as the noise level rises: from \(7.04 \times 10^{-7}\) in the noise-free case to \(7.26 \times 10^{-7}\) at 5\% noise. This stability suggests that the method is robust with respect to moderate noise when the discrepancy principle is used for stopping.

\begin{table}[h!]
\centering
\begin{tabular}{|c|c|c|}
\hline
\textbf{Noise Level (\%)} & \textbf{Error} & \textbf{Number of Iterations} \\
\hline
0.0 & \(7.04 \times 10^{-7}\) & 5 \\
1.0 & \(7.05 \times 10^{-7}\) & 5 \\
3.0 & \(7.11 \times 10^{-7}\) & 5 \\
5.0 & \(7.26 \times 10^{-7}\) & 5 \\
\hline
\end{tabular}
\caption{Error values and number of iterations for different noise levels (discrepancy-based stopping).}
\label{tab:error_noise3}
\end{table}

As in Example 1, the following figure displays the evolution of the recovered function with a noise level of $p = 1\%$. The solution becomes increasingly close to the exact function as the iterations progress.
Figure~\ref{fig:objective_decrease3} shows the evolution of the functional \( J_{\epsilon}(f_n) \) over the Conjugate Gradient iterations. The functional value decreases consistently from iteration~0 to~5, reflecting effective convergence. The logarithmic scale emphasizes the rapid drop at the beginning and the smoother decline in later steps.
\\
\begin{figure}[H]
    \centering
    \includegraphics[scale=0.4]{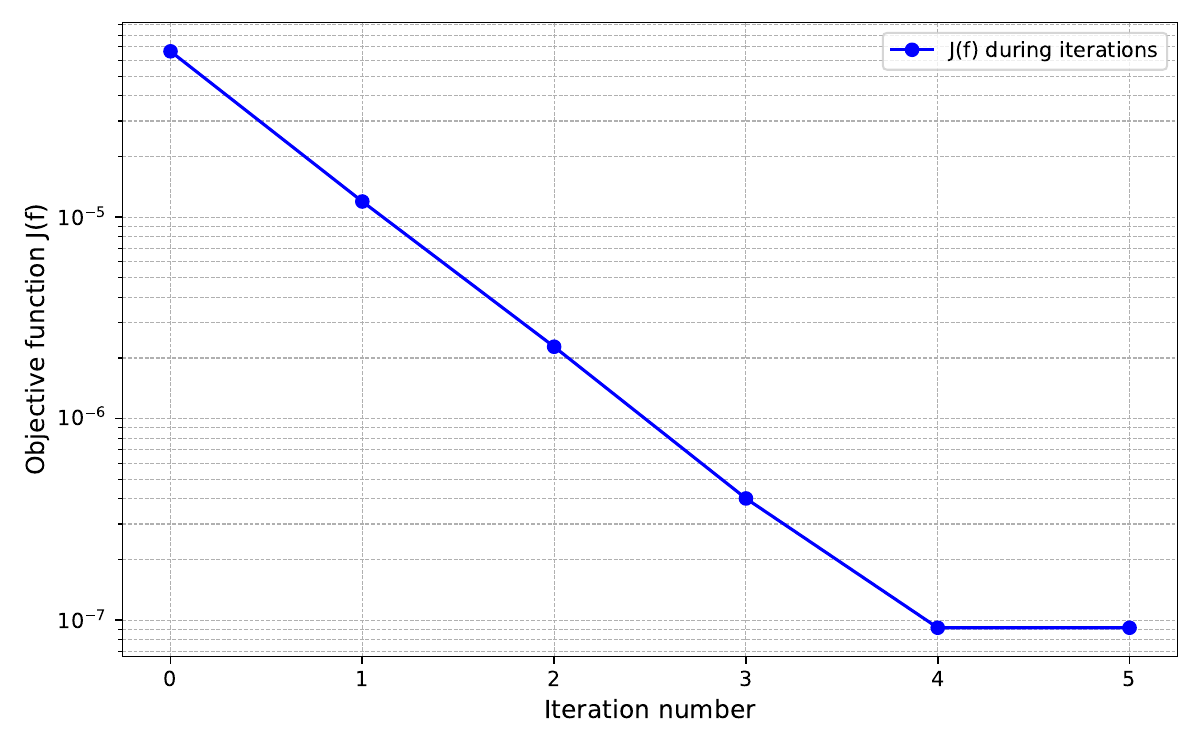}
    \caption{Decrease of the functional during the Conjugate Gradient process}
    \label{fig:objective_decrease3}
\end{figure}

\end{example}
\subsection{Degenerate viscous Hamilton–Jacobi equation}

Now we aim to determine the initial datum $f(x)$ of the following degenerate viscous Hamilton-Jacobi equation
\begin{equation}
    \left\{
\begin{aligned}
& u_{t} - a(x) u_{xx}+ \dfrac{b}{q}\vert u_{x}\vert^q + du_{x} +c u= 0,&&\quad \text{ in } Q,\\
&u(x,t)=0,&& \quad \text{ in }\Sigma,\\
&u(x,0)=f(x), &&\quad \text{ in } \Omega,
\end{aligned}\label{HJ}
\right.
\end{equation}
where $b, c, d\in L^{\infty}(Q)$, $q \geqslant 1$, and the degenerate coefficient $a(x)$ satisfies \textbf{Assumption I}.

As before, let us introduce the input-output operator $\Psi$ defined by
$\Psi(f)(x)=u(x,T;f), \; x\in (0,1),$
where $u(x,t;f)$ solves \eqref{HJ}. Then
the backward parabolic problem can be reformulated as the following operator equation:
\begin{align}
    \Psi(f)=u_{T}, \qquad f\in L^2_{\frac{1}{a}}(0,1). 
\end{align}
Let $u^{\delta}_T$ be a noisy measurement corresponding to the exact datum $u_T$.

Since the equation \eqref{HJ} is nonlinear, we propose an alternative approach for the identification problem instead of the adjoint state method presented in Section \ref{s1}. For simplicity, we take $d(x,t)=0$ and $c(x,t)=0$ for drift and potential coefficients and $q=2$ in the next numerical simulation.\\
\subsubsection{\textbf{Van Cittert iteration}}
Following \cite{Carasso}, we consider the Van Cittert iteration \cite{Van} given by
\begin{equation}\label{vancittert}
    f_{n+1}(x)=f_{n}(x)+\gamma \left(u_T^\delta -\Psi(f_{n})(x)\right), \qquad n\geqslant 1,
\end{equation}
where $\gamma>0$ is a relaxation parameter and $f_1(x)=\gamma u^{\delta}_T$ is the first iterate. The Van Cittert iteration is simple to implement. It was used in image restoration, and it is useful in applications such as astronomy, medical imaging, etc. We refer to \cite[Chapter 4]{LB91} for more details. This iteration has been adapted for some 2D nonlinear backward parabolic problems in \cite{Car13}.

In practice, measurements always contain noise; instead of $u_T$ we have $u_T^\delta$ such that 
$$\|u_T-u_T^\delta\|_\frac{1}{a}\leqslant \delta \qquad \text{for}\; \delta \geqslant 0.$$
Since the problem is ill-posed in general, a small perturbation in data can $u_T$ cause large errors in the solution. To tackle this challenge, the regularization was implemented by stopping the iterations as soon as the value of the residual error becomes just below the noise level; see, e.g., \cite{Alifanov2012,cao2019determination, cao2020simultaneous,Ozisik2000}. As a stopping criterion, we choose the first iteration $n$ such that
$$r_n=\|u_T^\delta-\Psi(f_n)\|\leqslant \epsilon'+\rho \mathcal{E},$$
where $\epsilon'$ is the tolerance imposed if there is no noise, $\mathcal{E}$ is the magnitude of noise, and $\rho$ is a constant strictly greater than $1$, which can be fixed at $1.01$.

\subsubsection{\textbf{Numerical experiments}}
We approximate the solution of the nonlinear equation \eqref{HJ} on the space-time domain $[0,1]\times [0,T]$ using the finite difference method. We discretize $[0,1]\times [0,T]$ into a uniform grid $x_{j}=jdx $ for $j=0,\cdots, N_x+1$, with a spatial step size $dx=\frac{1}{N_x+1} ,$ and $t_n=ndt$ for $n=0,\cdots , N_t+1,$ with a time step size $dt=\frac{T}{N_t+1}$. We denote by $ u^n_{j}$ the approximate solution at node  $x_{j} $ at time $t_n,$ and  $ a_{j}=a(x_{j}) $. The first-order and second-order derivatives of $u $ in \eqref{1} can be approximated  by:
$$u_t (x_j,t_n) \approx  \frac{u_{j}^{n+1}-u_{j}^{n}}{dt}$$
and 
$$u_{xx} (x_j,t_n) \approx \frac{1}{2}\bigg(\frac{u_{j+1}^{n+1}-2u_j^{n+1}+u_{j-1}^{n+1}}{dx^2}+\frac{u_{j+1}^{n}-2u_j^{n}+u_{j-1}^{n}}{dx^2} \bigg)$$
To approximate $u_x(x_j,t_n)$ in the nonlinear part, we use central differences
$$u_x(x_j,t_n)\approx  \frac{u^{n+1}_{j+1}-u^{n+1}_{j-1}}{2dx}.$$
Then, we obtain the following discrete problem:
\begin{empheq}[left = \empheqlbrace]{alignat=2}
\begin{aligned}\label{discsystem}
& \frac{u_{j}^{n+1}-u_{j}^{n}}{dt} - \frac{a_{j}}{2}\left(\frac{u_{j+1}^{n+1}-2u_j^{n+1}+u_{j-1}^{n+1}}{dx^2}+\frac{u_{j+1}^{n}-2u_j^{n}+u_{j-1}^{n}}{dx^2}\right)\\&\hspace{5cm} +\frac{b_j}{2}\left(\frac{u^{n+1}_{j+1}-u^{n+1}_{j-1}}{2dx}\right)^2 =0,\\
&\text{for} \; 1\leq j \leq N_x,\\
&u_0=u_{N_x+1}=0,
\end{aligned}
\end{empheq}
For $n\in \{1,2,\ldots,N_t\}$, we denote $u^n=\{u^n_j\}_{1\leqslant j\leqslant N_x}$. For each time step, we must solve a nonlinear system of the form
\begin{align}\label{nonlinequ}
    F(u^{n+1},u^n)=0,
\end{align}
where $F(\cdot,u^n): \mathbb{R}^{N_x}\longrightarrow \mathbb{R}^{N_x} $ is a nonlinear function that depends on the previous time step $u^n$, and $u^{n+1}\in \mathbb{R}^{N_x}$ denotes the unknown vector of interior values at the new time step. To solve \eqref{nonlinequ}, we use Newton's method.

           
\begin{algorithm}[H]
\caption{Algorithm with Newton’s Method}
\begin{algorithmic}[1]
    \State Given an initial guess $w_0,$ tolerance $\epsilon_1,$ maximum number of iterations $N_{\text{max}},$ and the previous time step $u^n$\;
    \State Set $k=0$\;
            \State \textbf{While} $k<N_{\text{max}}$ \textbf{do}
            \State \quad\textbf{if} $\vert F(w_k,u^n) \vert<\epsilon_1$ \textbf{then}
            \State \qquad\textbf{return} $w_k$
            \State \quad\textbf{end if}
            \State \quad Solve $\mathcal{J}\delta w_k=-F(w_k,u^n)$ with respect to $\delta w_k $, where $\mathcal{J}$ is the Jacobian matrix of $F(\cdot,u^n)$
            \State \quad $w_{k+1}\leftarrow w_k-\delta w_k$
            \State \quad $k\leftarrow k+1$
            \State \textbf{end while}
            \State \textbf{return $w_k$}
\end{algorithmic}
    
\end{algorithm}
 We set $\gamma=0.5$ for the relaxation parameter, $a(x)=(x(1-x))^4$ , $T=0.03$ and $N_x=N_t=100$ in the following numerical tests.

\begin{example}\label{ex3}
We take the exact initial datum as
$$f(x)=(x(1-x))^2\sin(\pi x), \qquad x\in (0,1),$$
and consider the following coefficient for the Hamiltonian term:
$$ b(x,t)=b(x)=(x(1-x))^2.$$ 
In the noise-free case, i.e., $p=0\%$; see Figure \ref{fig3}, the algorithm required only $n = 12 $ iterations to achieve a lower error. 

\begin{figure}[H]
    \centering
    {{\includegraphics[scale=0.5]{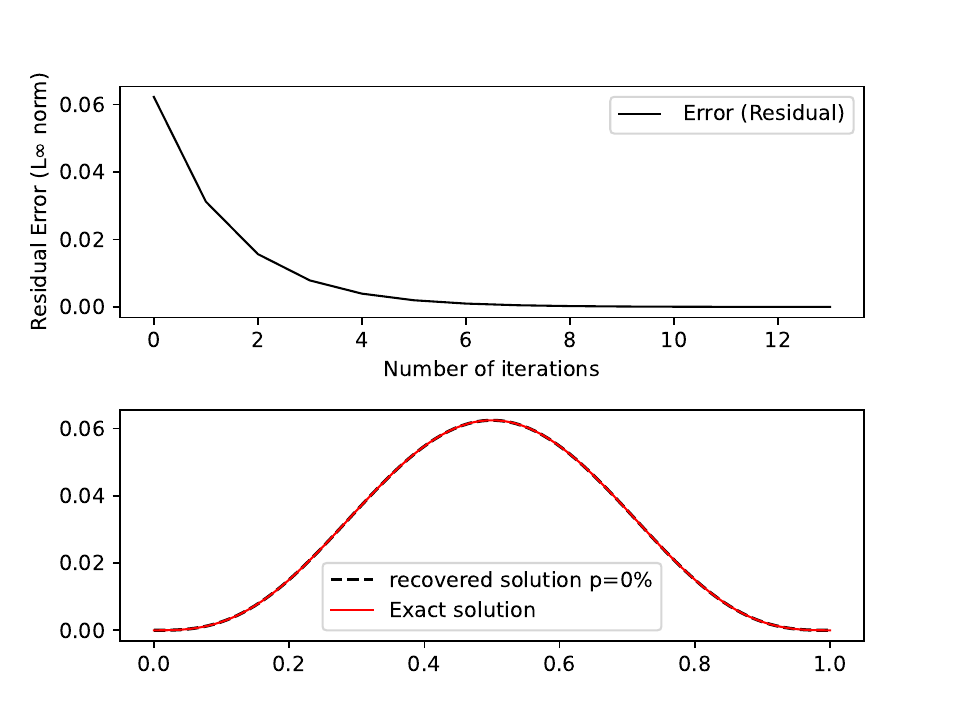} }}%
    \caption{(Bottom) Exact and recovered initial data after $12$ iterations from noise-free data. (Top) Residual error in terms of iterations}%
    \label{fig3}%
\end{figure}
\vspace{-0.5cm}
\begin{figure}[H]
\centering
\begin{subfigure}{0.5\textwidth}
    \centering
    \includegraphics[width=\textwidth]{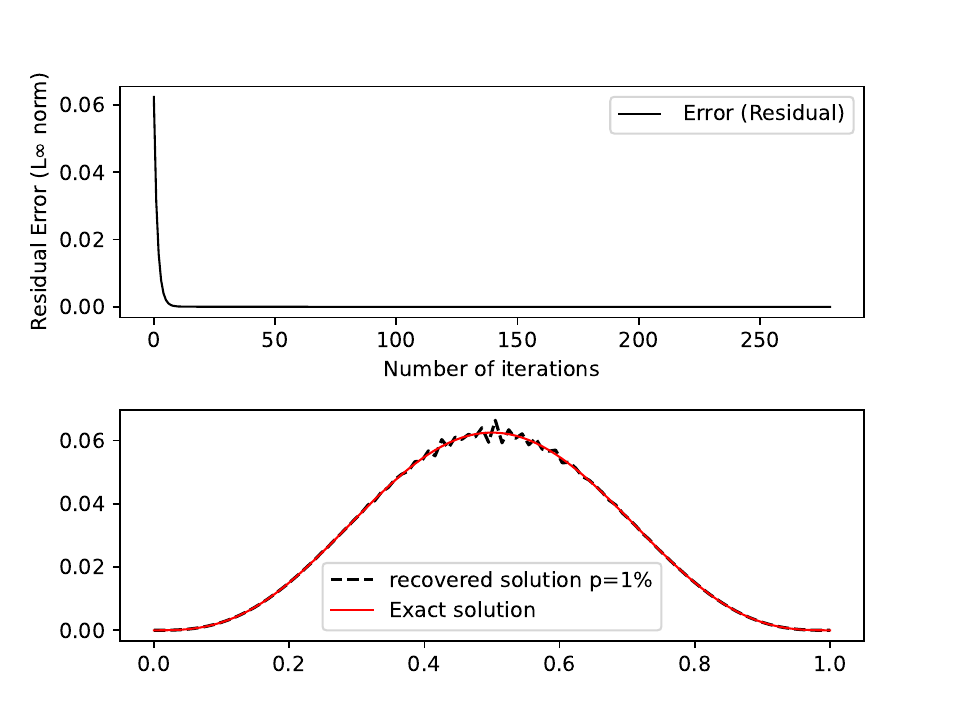}
    \caption{Exact and recovered initial data after $278$ iterations for $p=1 \%$ without regularization}
    \label{11NOISE}
\end{subfigure}
\hfill
\begin{subfigure}{0.48\textwidth}
    \centering
    \includegraphics[width=\textwidth]{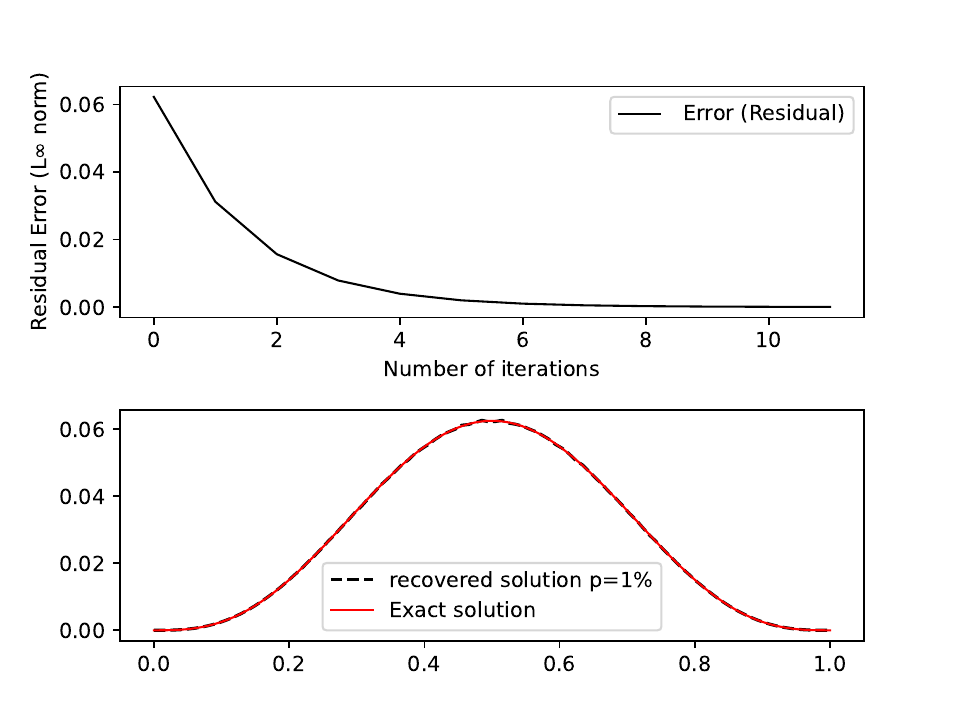}
    \caption{Exact and recovered initial data after $10$ iterations for $p=1 \%$ with early stopping}
    \label{12Noise}
\end{subfigure}
\caption{Comparison between the standard iteration and the early stopping for $p=1\%$}
\label{1 noise}
\end{figure}
\vspace{-0.5cm}
Figure \ref{1 noise} gives the comparison of the algorithm with and without regularization by early stopping for a noise level of $1\%$. As the process progresses, the residual error decreases but remains unstable due to noise (Figure \ref{11NOISE}). However, using an early stopping (Figure \ref{12Noise}) changes the situation. By interrupting the calculations after only $10$ iterations (i.e., as soon as the error falls below the noise level), the recovered solution is much closer to the exact solution.
\begin{figure}[H]
\centering

\begin{subfigure}{0.48\textwidth}
    \centering
    \includegraphics[width=\textwidth]{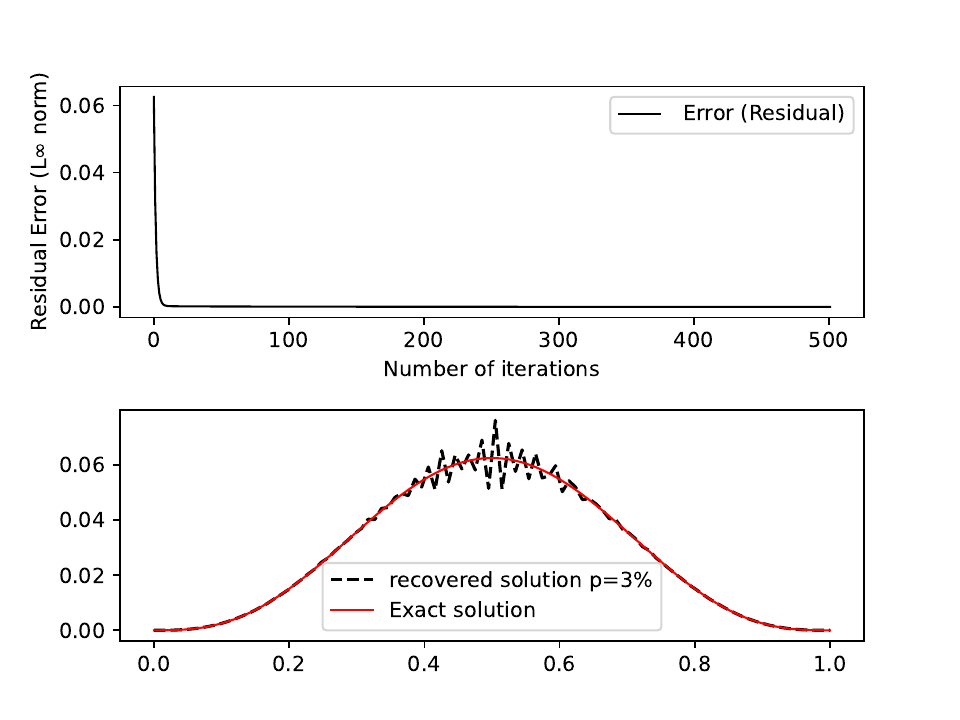}
    \caption{Exact and recovered initial data after $500$ iterations for $p=3 \%$ without regularization}
    \label{noise 111}
\end{subfigure}
\hfill
\begin{subfigure}{0.48\textwidth}
    \centering
    \includegraphics[width=\textwidth]{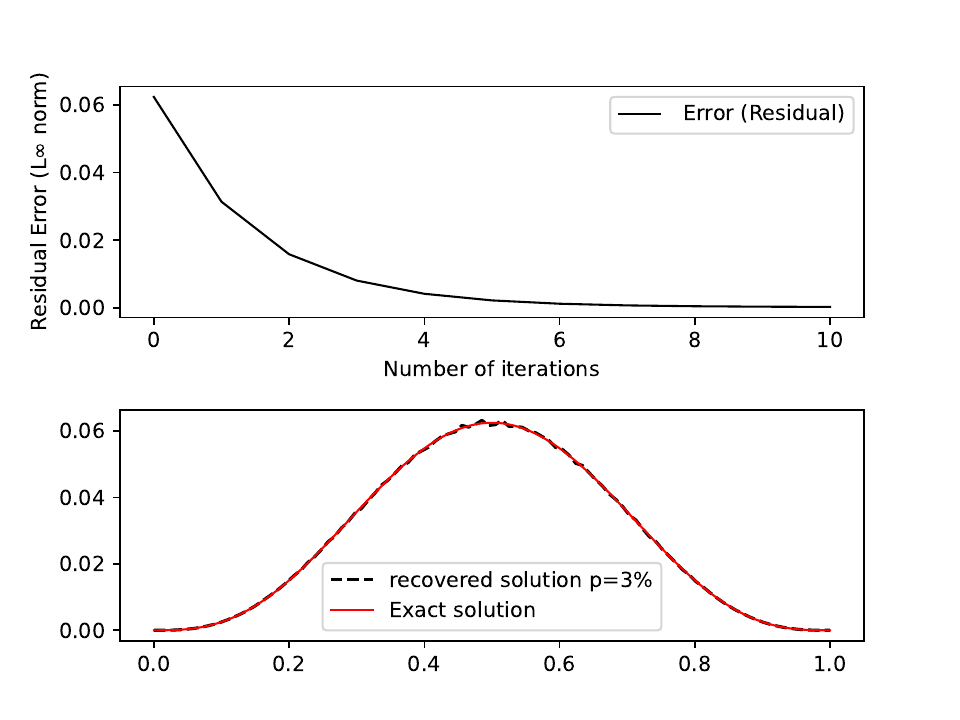}
    \caption{Exact and recovered initial data after $10$ iterations for $p=3 \%$ with early stopping }
\end{subfigure}
\caption{Comparison between standard iteration and early stopping for $p=3\%$}
\label{3 noise}
\end{figure}
 For a noise level of $3\%$; see Figure \ref{3 noise}, the behavior remains the same. Without regularization, the obtained solution shows significant numerical instability (see Figure \ref{noise 111}).
\end{example}

\begin{remark}
The aforementioned numerical tests indicate the usefulness of the Van Cittert method after a finite number of iterations, despite its possible lack of convergence.
\end{remark}

 \section{Conclusion and final comments}\label{sec6}
The backward problem of determining unknown initial data in degenerate viscous Hamilton-Jacobi equations from final-time measurements is of significant interest in various applied science applications. While models with nondegenerate diffusion coefficients have been extensively studied within well-established theoretical frameworks, much less attention has been given to viscous Hamilton-Jacobi equations involving degenerate diffusion coefficients. In such cases, both the theory and numerical methods remain at an early stage due to several substantial challenges.

First, the degeneracy of the diffusion coefficient makes the well-posedness and regularity of the viscous Hamilton-Jacobi equations significantly more subtle. This issue has been successfully addressed in our paper by working in appropriately weighted $L^2$ spaces and using norms that account for the degeneracy. On the numerical side, solving degenerate equations is more delicate, as it requires refined schemes for both discretization and integration. In our setting, the boundary degeneracy, combined with Dirichlet boundary conditions, allowed us to overcome these difficulties. However, the case of interior degeneracy or Neumann boundary conditions would be a worthwhile direction for further investigation.

Regarding the backward problem associated with the nonlinear viscous Hamilton-Jacobi equation \eqref{23}, we have proven the stability estimate (Theorem \ref{stabHJ}) for the Hamiltonian exponent $q\ge 1$ thanks to inequality \eqref{kineq}. Therefore, it is natural to ask whether one can prove similar stability results for the case $0<q<1$.

Finally, our analysis has been carried out for the one-dimensional interval $(0,1)$ in the weighted space $L^2_{\frac{1}{a}}(0,1)$, where a nondivergence form operator has been studied naturally. Extending our results to a multi-dimensional domain $\Omega \subset \mathbb{R}^d$ ($d\ge 2$ is an integer) remains an interesting direction for future research. The Carleman estimate approach we developed is promising, but establishing the appropriate functional framework and conducting the numerical analysis in higher dimensions are more subtle.

\end{document}